\documentclass[11pt,reqno]{amsart}

\usepackage{amsthm, amsmath, amsfonts, amssymb, graphicx, tikz-cd, MnSymbol, comment, hyperref, enumerate}

\usepackage[utf8]{inputenc}

\DeclareFontFamily{U}{mathb}{\hyphenchar\font45}
\DeclareFontShape{U}{mathb}{m}{n}{
<-6> mathb5 <6-7> mathb6 <7-8> mathb7
<8-9> mathb8 <9-10> mathb9
<10-12> mathb10 <12-> mathb12
}{}
\DeclareSymbolFont{mathb}{U}{mathb}{m}{n}
\DeclareMathSymbol{\llcurly}{\mathrel}{mathb}{"CE}

\newtheorem{theorem}{Theorem}[section]
\newtheorem{lemma}[theorem]{Lemma}
\newtheorem{proposition}[theorem]{Proposition}
\newtheorem{corollary}[theorem]{Corollary}

\newtheorem{claim}[theorem]{Claim}
\theoremstyle{definition}
\newtheorem{definition}[theorem]{Definition}
\newtheorem{notation}[theorem]{Notation}

\newtheorem{remark}[theorem]{Remark}

\newenvironment{proofclaim}{\paragraph{\emph{Proof of the Claim}.}}{\hfill$\qed$\\}

\newcommand{\func}[1]{\operatorname{#1}}

\def\Sp{{\sf Sp}}

\def\ann{{\sf Ann}}
\def\arch{{\sf Arch}}

\def\RO{\mathcal{RO}}

\newcommand\End{{\sf End}}
\newcommand{\dv}[1]{\mathfrak{#1}}

\def\bal{\boldsymbol{\mathit{ba}\ell}}
\def\ubal{\boldsymbol{\mathit{uba}\ell}}
\def\dbal{\boldsymbol{\mathit{dba}\ell}}

\def\pda{{\sf PDA}}
\def\cpda{{\sf KT}}
\def\KHaus{{\sf KHaus}}
\def\KRFrm{{\sf KRFrm}}
\def\dev{{\sf DeV}}

\def\PBSp{{\sf PBSp}}

\def\cl{{\sf cl}}
\def\int{{\sf int}}

\newcommand{\Id}{\func{Id}}
\newcommand{\spec}{\mathcal{S}}

\begin{document}

\title{A Unified Approach to Gelfand and de Vries Dualities}

\author{G.~Bezhanishvili}
\address{New Mexico State University}
\email{guram@nmsu.edu}

\author{L.~Carai}
\address{Universit\`a degli Studi di Salerno}
\email{lcarai@unisa.it}

\author{P.~J.~Morandi}
\address{New Mexico State University}
\email{pmorandi@nmsu.edu}

\author{B.~Olberding}
\address{New Mexico State University}
\email{bruce@nmsu.edu}


\date{}

\begin{abstract}
We develop a unified approach to Gelfand and de Vries dualities for compact Hausdorff spaces, which is based on appropriate modifications of the classic results of Dieudonn\'{e} (analysis), Dilworth (lattice theory), and Kat{\v{e}}tov-Tong (topology).
\end{abstract}

\maketitle

\tableofcontents

\section{Introduction}

Compact Hausdorff spaces enjoy several algebraic, analytic, and lattice-theoretic representations, which are at the heart of duality theory for the category $\KHaus$ of compact Hausdorff spaces and continuous maps. One of the oldest such is known under the name of Gelfand duality (see, e.g., \cite[Ch.~IV.4]{Joh82}), and can be presented in various signatures, depending on whether we work with real-valued or complex-valued functions (see, e.g., \cite{BMO13a} and the references therein). We will follow the standard practice in topology and work with continuous real-valued functions on $X\in\KHaus$. This gives rise to the lattice-ordered algebra $C(X)$ which is bounded (because $X$ is compact) and archimedean (because there are no infinitesimals in $\mathbb R$). In addition, $C(X)$ is uniformly complete in the norm topology. As a result, we arrive at the category $\bal$ of bounded archimedean $\ell$-algebras and (unital) $\ell$-algebra homomorphisms, and its reflective subcategory $\ubal$ consisting of uniformly complete objects in $\bal$. Gelfand duality then yields a dual adjunction between $\KHaus$ and $\bal$ which restricts to a dual equivalence between $\KHaus$ and $\ubal$ (see Theorem~\ref{thm: Gelfand}).

If instead of real-valued functions we work with regular open subsets of $X$, we arrive at de Vries duality \cite{deV62} between compact Hausdorff spaces and what later became known as de Vries algebras \cite{Bez10}. These are complete boolean algebras equipped with a binary relation that captures the proximity relation on the complete boolean algebra $\RO(X)$ of regular open subsets of $X$ given by $U\prec V$ iff $\cl(U)\subseteq V$. De Vries duality then yields a dual equivalence between $\KHaus$ and the category $\dev$ of de Vries algebras and de Vries morphisms (see Theorem~\ref{thm: de Vries}). 

Both de Vries and Gelfand dualities were generalized in several directions. In \cite{BMO19a, BMO20b} both dualities were extended to completely regular spaces and their compactifications. In \cite{DH18} Gelfand duality was generalized to the setting of compact ordered spaces studied by Nachbin \cite{Nac65}, and in \cite{DIT22} a general categorical framework was developed that yields de Vries duality and its generalizations. However, as far as we know, there is no unifying approach to Gelfand and de Vries dualities. Our aim is to develop such an approach, the key ingredients of which are based on appropriate modifications of classic results of Dieudonn\'{e}, Dilworth, and Kat{\v{e}}tov-Tong.

To begin, we can define a functor from $\bal$ to $\dev$ using the theory of annihilator ideals. 
We recall (see, e.g., \cite[Rem.~4.2(1)]{BCM21d}) that kernels of $\bal$-morphisms are archimedean $\ell$-ideals (see Definition~\ref{def: arch}); that is, $\ell$-ideals $I$ on $A\in\bal$ such that $A/I\in\bal$. If $A=C(X)$, these ideals correspond to open subsets of $X$. As we will see in Section~\ref{sec:ann}, regular opens of $X$ correspond to annihilator ideals of $C(X)$, and this gives rise to a covariant functor $\bal\to\dev$ which associates to each $A\in\bal$ the de Vries algebra of annihilator ideals of $A$. 

Going from $\dev$ to $\bal$ is less obvious, and will require several nontrivial steps. As the first step, we find an $\ell$-algebra that contains both $C(X)$ and $\RO(X)$. This is closely related to Dilworth's characterization \cite{Dil50} of the Dedekind completion of $C(X)$. Let $B(X)$ be the $\ell$-algebra of bounded real-valued functions on $X$. We recall (see, e.g., \cite[Sec.~2]{Dan15}) that the Baire operators on $B(X)$ are defined by
\[
f_*(x)=\sup_{U\in\mathcal N_x}\inf_{y\in U}f(y) \mbox{ and } f^*(x)=\inf_{U\in\mathcal N_x}\sup_{y\in U}f(y),
\]
where $f\in B(X)$, $x\in X$, and $\mathcal N_x$ is the family of open neighborhoods of $x$. A function $f\in B(X)$ is {\em lower-semicontinuous} if $f=f_*$ and {\em upper-semicontinuous} if $f=f^*$. We say that a lower-semicontinuous function $f$ is \emph{normal} if $f=(f^*)_*$. Let $N(X)$ be the set of normal functions on $X$. Then $N(X)$ is an $\ell$-algebra, where the $\ell$-algebra operations on $N(X)$ are normalizations of the $\ell$-algebra operations on $B(X)$ (see Remark~\ref{rem: N(X)}). Dilworth \cite{Dil50} proved that if we view $C(X)$ and $N(X)$ as lattices, then $N(X)$ is isomorphic to the Dedekind completion of $C(X)$. Later D{\u{a}}ne{\c{t}} \cite{Dan15} showed that $N(X)$ remains isomorphic to the Dedekind completion of $C(X)$ in the richer signature of vector lattices, and it follows from \cite[Sec.~8]{BMO16} that this also remains true in the signature of $\ell$-algebras. Thus, we can phrase a strengthened version of Dilworth's theorem as follows:

\begin{theorem} [Dilworth's theorem]
If $X\in\KHaus$, then $N(X)$ is isomorphic to the Dedekind completion of $C(X)$ in $\bal$.
\end{theorem}

We can recover $C(X)$ from $N(X)$ by utilizing the celebrated Kat{\v{e}}tov-Tong theorem in topology.

\begin{theorem} [Kat{\v{e}}tov-Tong]
Let $X$ be a normal space and $f,g\in B(X)$ such that $f^* \le g_*$. Then there is $h\in C(X)$ with $f^*\le h\le g_*$.
\end{theorem}

Since each compact Hausdorff space is normal, the Kat{\v{e}}tov-Tong theorem is available in our context. Thus, we can define a proximity relation $\lhd$ on $N(X)$ by setting $f\lhd g$ iff $f^*\le g$ (note that $g = g_*$), and use the Kat{\v{e}}tov-Tong theorem to recover $C(X)$ as the $\ell$-algebra of reflexive elements. Because of this connection, we call such proximity relations on $N(X)$ {\em Kat{\v{e}}tov-Tong proximities} or {\em KT-proximities} for short (see Definition~\ref{def: proximity dedekind algebra}). 

To connect $N(X)$ to $\RO(X)$, we point out that the idempotents of the ring $N(X)$ are exactly the characteristic functions of the regular open subsets of $X$. Consequently, both $C(X)$ and $\RO(X)$ live inside $N(X)$. Namely, $C(X)$ is the $\ell$-algebra of reflexive elements of the KT-proximity on $N(X)$ while the idempotents of $N(X)$ are the characteristic functions arising from $\RO(X)$. 

It is natural to consider the $\ell$-subalgebra of $N(X)$ generated by its idempotents. Such algebras are related to the theory of the Baer-Specker group and its subgroups (see \cite{BMO20e} and the references therein). Because of this, they were named {\em Specker algebras} in \cite{BMO13a}. As follows from \cite{BMMO15b}, the Specker subalgebra of $N(X)$ is exactly the $\ell$-algebra $FN(X)$ of finitely-valued normal functions on $X$. Moreover, the de Vries proximity on $\RO(X)$ lifts to a proximity on $FN(X)$. Furthermore, $N(X)$ is the Dedekind completion of $FN(X)$, and there is a natural lift of the proximity on $FN(X)$ to $N(X)$. The last step is to show that this lift coincides with the KT-proximity on $N(X)$. This requires Dieudonn\'{e}'s lemma, which is our last ingredient. This lemma is more of a proof-technique which originates in \cite{Die44}, and was used by various authors in different contexts (see, e.g., \cite{Edw66,BS75,Lan76,BMO20c}). We will require it in the following form:

\begin{theorem} [Dieudonn\'{e}'s lemma]
Let $X \in \KHaus$ and $\lhd$ be a proximity on $FN(X)$. Then the closure of $\lhd$ is a KT-proximity on $N(X)$.
\end{theorem}

It is this lemma that allows us to show that the lift of the proximity on $FN(X)$ to $N(X)$ is the KT-proximity on $N(X)$. Thus, we can go from $\RO(X)$ to $N(X)$ through the Specker algebra $FN(X)$. Since the boolean algebra of idempotents of $FN(X)$ is isomorphic to the complete boolean algebra $\RO(X)$, we have that $FN(X)$ is a Baer ring (see Section~\ref{sec: pda and dev}). We first lift the de Vries proximity on $\RO(X)$ to a proximity on the Baer-Specker algebra $FN(X)$, and then use Dieudonn\'{e}'s lemma to show that the lift of the proximity on $FN(X)$ is the KT-proximity on $N(X)$. Moreover, $C(X)$ can be recovered as the reflexive elements of the KT-proximity on $N(X)$. As a result, we arrive at the following diagram, which commutes up to natural isomorphism (see Section~\ref{sec: summary}):

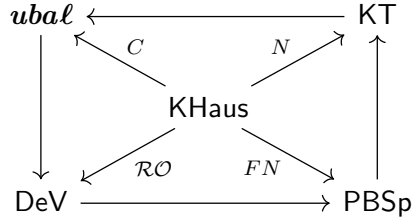
\[
\begin{tikzcd}
\ubal \arrow[dd] && \cpda \arrow[ll] \\
& \KHaus \arrow[ul, "C"'] \arrow[ur, "N"] \arrow[dl, "\RO"] \arrow[dr, "FN"'] & \\
\dev \arrow[rr] && \PBSp \arrow[uu] 
\end{tikzcd}
\]
Here $\cpda$ is the category of what we term Kat{\v{e}}tov-Tong algebras; that is, Dedekind algebras equipped with a KT-proximity that is closed in the product topology (see Definition~\ref{def: category KT}). Also, $\PBSp$ is the category of proximity Baer-Specker algebras of \cite{BMMO15b} (see Section~\ref{sec: proximity Baer-Specker}). Each of the four categories $\ubal$, $\dev$, $\PBSp$, and $\cpda$ is dually equivalent to $\KHaus$.
That $\KHaus$ is dually equivalent to $\ubal$ is Gelfand duality, and that $\KHaus$ is dually equivalent to $\dev$ is de Vries duality. The dual equivalence of $\KHaus$ and $\PBSp$ is established in \cite{BMMO15b}, and the dual equivalence of $\KHaus$ and $\cpda$ in \cite{BMO16}. Consequently, the four categories $\ubal$, $\dev$, $\PBSp$, and $\cpda$ are equivalent. However, these equivalences are obtained by utilizing duality theory for $\KHaus$, and hence require, among other things, the use of the axiom of choice. We give a direct and choice-free proof of each of these four equivalences. 

In Section~\ref{sec:ann} we describe the functor $\ann : \ubal \to \dev$ which associates with each $A\in\ubal$ the de Vries algebra of annihilator ideals of $A$. In Section~\ref{sec: Dieudonne} we prove that $\ubal$ is equivalent to $\cpda$. This is done by first establishing an appropriate version of Dieudonn\'{e}'s lemma, which is our first main result. In Section~\ref{sec: pda and dev} we define the functor $\Id : \cpda \to \dev$ which associates with each KT-algebra the de Vries algebra of its idempotents, and in Section~\ref{sec: proximity Baer-Specker} we describe the functors establishing an equivalence between $\dev$ and $\PBSp$. Finally, in Section~\ref{sec:BS and KT} we prove that $\cpda$ is equivalent to both $\dev$ and $\PBSp$, which is our second main result. This completes our proof that the four categories in the diagram are equivalent, thus yielding a unified approach to Gelfand and de Vries dualities.

Establishing these category equivalences requires a number of intricate arguments, many of which are given in the course of the article, while others 
are cited from some of our previous articles. One of the lengthier and most technical arguments is a proof that weak proximity morphisms between proximity Baer-Specker algebras are in fact proximity morphisms.  We  have placed this proof in an appendix since although this lemma is essential for us, the sequence of ideas used in proving it is not  needed to follow the main ideas.

\section{Gelfand and de Vries dualities}

As we saw in the introduction, with each $X\in\KHaus$ we can associate the $\ell$-algebra $C(X)$ of continuous real-valued functions on $X$ and the de Vries algebra  $\RO(X)$ of regular open subsets of $X$. The first approach leads to Gelfand duality and the second to de Vries duality. In this section we briefly recall these dualities.

We start with Gelfand duality. All algebras we will consider are commutative and unital (have 1). With respect to pointwise operations, $C(X)$ is a lattice-ordered algebra or an $\ell$-algebra for short, where we recall that $A$ is an {\em $\ell$-algebra} if $A$ is an $\mathbb R$-algebra and a lattice such that for all $a,b,c\in A$ and $r\in\mathbb R$ we have:
\begin{itemize}
\item $a\le b$ implies $a+c\le b+c$;
\item $0\le a$ and $0\le b$ imply $0\le ab$; 
\item $0\le a$ and $0\le r$ imply $0\le r \cdot a$.
\end{itemize}
Moreover, since $X$ is compact, $C(X)$ is bounded, and since $\mathbb R$ has no infinitesimals, $C(X)$ is archimedean, where we recall that an $\ell$-algebra $A$ is
\begin{itemize}
\item \emph{bounded} if for each $a \in A$ there is an integer $n \ge 1$ such that $a \le n\cdot 1$ (that is, $1$ is a \emph{strong order unit}); and
\item it is \emph{archimedean} if for each $a, b \in A$, whenever $n\cdot a \le b$ for each $n \ge 1$, then $a \le 0$.
\end{itemize}
This motivates the following definition (see \cite[Sec.~2]{BMO13a}):

\begin{definition}\label{def: bal}
A \emph{$\bal$-algebra} is a bounded archimedean $\ell$-algebra and a $\bal$-morphism is a unital $\ell$-algebra homomorphism. Let $\bal$ be the category of $\bal$-algebras and $\bal$-morphisms.
\end{definition}

Let $A\in\bal$. Since $A$ is a bounded $\ell$-algebra, it is an $f$-ring (see, e.g., \cite[Lem.~XVII.5.2]{Bir79}), meaning that if $0 \le a, b, c \in A$ with $a \wedge b = 0$, then $(ac) \wedge b = 0$. For each $a\in A$ we can define the \emph{positive} and \emph{negative} parts of $a$ by
\[
a^+ = a \vee 0 \textrm{ and }  a^- = (-a) \vee 0 = -(a \wedge 0).
\]
Then $a = a^+ - a^-$ and $a^+ \wedge a^- = a^+  a^- = 0$ (see, e.g., \cite[Eqn.~XIII.3(15), Thm.~XIII.4.7, Lem.~XVII.5.1]{Bir79}). We define the \emph{absolute value} of $a$ by
\[
|a|=a\vee(-a),
\]
and the \emph{norm} of $a$ by
\[
||a||=\inf\{r\in\mathbb R : |a|\le r\}.\footnote{When $A \ne 0$ we view $\mathbb{R}$ as an $\ell$-subalgebra of $A$ by identifying $r \in \mathbb{R}$ with $r \cdot 1 \in A$.}
\]

\begin{definition}
We call $A$ \emph{uniformly complete} if the norm is complete. Let $\ubal$ be the full subcategory of $\bal$ consisting of uniformly complete objects.
\end{definition}

It is easy to see that if $X\in\KHaus$, then $C(X)\in\ubal$, where for $f\in C(X)$ we have
\[
||f||=\sup\{ |f(x)| : x\in X \}.
\]
This defines a contravariant functor $C:\KHaus \to \ubal$ which associates with each $X\in\KHaus$ the $\ell$-algebra $C(X)$ of (necessarily bounded) continuous real-valued functions on $X$; and with each continuous map $\varphi:X\to Y$ the $\ell$-algebra homomorphism $C(\varphi):C(Y)\to C(X)$ given by $C(\varphi)(f)=f\circ\varphi$ for each $f\in C(Y)$. 

To define the contravariant functor $\bal \to \KHaus$, we recall the notion of an {\em $\ell$-ideal}; that is, an ideal $I$ of $A\in\bal$ such that $|a|\le|b|$ and $b\in I$ imply $a\in I$. The {\em Yosida space} $Y(A)$ of $A\in\bal$ is the set of maximal $\ell$-ideals of $A$ whose closed sets are exactly sets of the form
\[
Z_\ell(I) = \{M\in Y(A): I\subseteq M\},
\]
where $I$ is an $\ell$-ideal of $A$. It is well known that $Y(A)\in\KHaus$. This defines a contravariant functor $Y:\bal\to\KHaus$ which sends $A\in\bal$ to its Yosida space $Y(A)$, and a $\bal$-morphism $\alpha:A\to A'$ to $Y(\alpha)=\alpha^{-1}:Y(A')\to Y(A)$.

The functors $C$ and $Y$ yield a dual adjunction between $\KHaus$ and $\bal$.
Moreover, for $X\in\KHaus$ we have that $\varepsilon_X: X\to Y(C(X))$
is a homeomorphism, where
\[
\varepsilon_X(x)=\{f\in C(X) : f(x)=0\}.
\]
Furthermore, for $A\in\bal$ and $M$ a maximal $\ell$-ideal of $A$, it is well known (see, e.g., \cite[Cor.~27]{HJ61}) that $A/M \cong \mathbb{R}$. Therefore, we can define $\zeta_A :A\to C(Y(A))$ by $\zeta_A(a)(M)=r$ where $r$ is the unique real number satisfying $a+M=r+M$. Then $\zeta_A$ is a monomorphism in $\bal$ separating points of $Y(A)$. Thus, by the Stone-Weierstrass theorem, we have that if $A$ is uniformly complete, then $\zeta_A$ is an isomorphism.
Consequently, the dual adjunction restricts to a dual equivalence between $\ubal$ and $\KHaus$, yielding Gelfand duality:

\begin{theorem} [Gelfand duality \cite{GN43,Sto40}] \label{thm: GD} \label{thm: Gelfand}
The contravariant functors $C$ and $Y$ yield a dual adjunction between $\KHaus$ and $\bal$ which restricts to a dual equivalence between $\KHaus$ and $\ubal$.
\end{theorem}

We next turn to de Vries duality \cite{deV62}. For a boolean algebra $B$ and $a \in B$ we write $a^*$ for the complement of $a$ in $B$. A \emph{de Vries algebra} is a pair $B = (B,\prec)$ consisting of a complete boolean algebra $B$ together with a binary relation $\prec$ satisfying

\begin{enumerate} \label{DV properties}
 \item[(DV1)] $1\prec 1$;
 \item[(DV2)] $a\prec b$ implies $a\leq b$;
 \item[(DV3)] $a\leq b\prec c\leq d$ implies $a\prec d$;
 \item[(DV4)] $a\prec b,c$ implies $a\prec b\wedge c$;
 \item[(DV5)] $a\prec b$ implies $b^*\prec a^*$;
 \item[(DV6)] $a\prec b$ implies there is $c\in A$ with $a\prec c\prec b$;
 \item[(DV7)] $a\neq 0$ implies there is $b\neq 0$ with $b\prec a$.
\end{enumerate}

Given two de Vries algebras $B$ and $B'$, a \emph{de Vries morphism} is a map $\sigma : B \to B'$ satisfying 
\begin{enumerate}
 \item[(M1)] $\sigma(0)=0$;
 \item[(M2)] $\sigma(a\wedge b)=\sigma(a)\wedge\sigma(b)$;
 \item[(M3)] $a\prec b$ implies $\sigma(a^*)^*\prec\sigma(b)$;
 \item[(M4)] $\sigma(a)=\bigvee\{\sigma(b) : b\prec a\}$.
\end{enumerate}
For two de Vries morphisms $\sigma_1: B_1 \to B_2$ and $\sigma_2 : B_2 \to B_3$, the composition is given by 
\[
(\sigma_2\star\sigma_1)(a)=\bigvee\{\sigma_2\sigma_1(b) : b\prec a\}.
\] 

\begin{definition}
Let $\dev$ be the category of de Vries algebras and de Vries morphisms.
\end{definition}

Typical examples of de Vries algebras are the complete boolean algebras $\RO(X)$ of regular open subsets of $X\in\KHaus$ equipped with the binary relation $\prec$ given by 
\[
U\prec V \mbox{ iff } \cl(U)\subseteq V.
\]
Also, typical examples of de Vries morphisms are the maps $\RO(\varphi):\RO(Y)\to\RO(X)$ where $\varphi:X\to Y$ is a continuous map between compact Hausdorff spaces and 
\[
\RO(\varphi)(U)=\int(\cl\,\varphi^{-1}(U))
\]
for each $U\in\RO(Y)$. This defines a contravariant functor $\RO:\KHaus\to\dev$. 

To define a contravariant functor $\dev\to\KHaus$ we recall the notions of round filters and ends. Let $(B, \prec) \in\dev$.
For $S\subseteq B$ let 
\[
{\twoheaduparrow}S=\{a\in B:\exists s\in S \mbox{ with } s\prec a\}.
\]  
We call a filter $F$ of $B$ \emph{round} if $F={\twoheaduparrow}F$.  
Maximal round filters of $B$ are called \emph{ends}. 
Let $\mathcal E(B)$ be the set of ends of $B$ topologized by the basis $\{\varepsilon(a): a\in B\}$, where 
\[
\varepsilon(a)=\{E\in\mathcal E(B): a\in E\}.
\]
Then $\mathcal E(B)$ is compact Hausdorff. For a de Vries morphism $\sigma:B\to B'$, let $\mathcal E(\sigma):\mathcal E(B')\to\mathcal E(B)$ be given by 
\[
\mathcal E(\sigma)(E)={\twoheaduparrow}\sigma^{-1}(E)
\]
for each $E\in\mathcal E(B')$. Then $\mathcal E(\sigma):\mathcal E(B')\to\mathcal E(B)$ is continuous. This gives rise to a contravariant functor $\mathcal E:\dev\to\KHaus$. The functors $\RO$ and $\mathcal E$ yield de Vries duality:

\begin{theorem} [De Vries duality \cite{deV62}] \label{thm: de Vries}
$\dev$ is dually equivalent to $\KHaus$.
\end{theorem}

\section{The annihilator ideal functor}\label{sec:ann}

In this section we show that there is a rather natural covariant functor from $\bal$ to $\dev$. This functor is obtained by working with annihilator ideals of $\bal$-algebras. We show that this is a functor by  viewing  annihilator ideals as archimedean $\ell$-ideals.

\begin{definition} \label{def: arch}
 Let $A \in \bal$. We call an $\ell$-ideal $I$ of $A$ \emph{archimedean} if $A/I$ is archimedean (equivalently $A/I\in\bal$).  Let $\arch(A)$ be the set of archimedean $\ell$-ideals of $A$, ordered by inclusion.
\end{definition}

\begin{remark} \label{rem: properties of arch} 
Let $A\in\bal$. If $M$ is a maximal $\ell$-ideal of $A$, then $A/M \cong \mathbb{R}$. Thus, every maximal $\ell$-ideal is archimedean. In fact, an $\ell$-ideal $I$ of $A \in \bal$ is archimedean iff $I = \bigcap \{ M \in Y(A) : I \subseteq M\}$ (see, e.g., \cite[p.~440]{BMO13a}).
\end{remark}

\begin{remark}
In \cite{Ban97} Banaschewski studied the $\ell$-ideals in bounded archimedean $f$-rings that are closed in the norm topology. If $A$ is a $\bal$-algebra, then an $\ell$-ideal $I$ of $A$ is archimedean iff it is closed in the norm topology. 
\end{remark}

It is a consequence of a more general result of Banaschewski \cite[App.~2]{Ban97} that $\arch(A)$ ordered by inclusion is a frame, where we recall (see, e.g., \cite{PP12}) that a \emph{frame} is a complete lattice $L$ satisfying the \emph{join infinite distributive law} 
\[
a\wedge\bigvee S=\bigvee\{a\wedge s : s\in S\}.
\] 
The meet in $\arch(A)$ is set-theoretic intersection and the join is the archimedean $\ell$-ideal generated by the union.  

We further recall that a frame $L$ is \emph{compact} if $\bigvee S=1$ implies $\bigvee T=1$ for some finite $T \subseteq S$. For $a\in L$ let 
$
a^*=\bigvee\{b\in L: a\wedge b=0\}
$ 
be the pseudocomplement of $a$, and for $a,b\in L$ define the \emph{well-inside} relation by 
\[
a\prec b \mbox{ iff } a^*\vee b=1.
\] 
Then a frame $L$ is \emph{regular} if for each $a\in L$ we have $a=\bigvee\{b\in L : b\prec a\}$.

Given two frames $L$ and $M$, a map $h:L\to M$ is a \emph{frame homomorphism} if $h$ preserves finite meets and arbitrary joins. 

\begin{definition}
Let $\KRFrm$ be the category of compact regular frames and frame homomorphisms. 
\end{definition}

It follows from Banaschewski's result \cite[App.~2]{Ban97} that $\arch(A) \in \KRFrm$. 
 Moreover, if $\alpha:A\to A'$ is a $\bal$-morphism, then $\arch(\alpha):\arch(A)\to\arch(A')$ is a frame homomorphism, where $\arch(\alpha)$ sends each $I\in\arch(A)$ to the archimedean $\ell$-ideal of $A'$ generated by $\alpha[I]$. 
Thus, as a consequence of Banaschewski's results, we obtain:

\begin{proposition}
\label{thm: arch frame}
$\arch : \bal \to \KRFrm$ is a covariant functor. 
\end{proposition}

As was observed in \cite{Bez12}, $\KRFrm$ is equivalent to $\dev$. We recall that an element $a$ of a frame $L$ is {\em regular} if $a^{**}=a$. The {\em booleanization} $\mathfrak B(L)$ of $L$ is the frame of regular elements of $L$. It is well known that $\mathfrak B(L)$ is a complete boolean algebra, where the meet and (pseudo)complement
in $\mathfrak B(L)$ are calculated as in $L$ and the join is calculated by the formula $\bigsqcup S=(\bigvee S)^{**}$. 

If $L\in\KRFrm$, then restricting the well-inside relation $\prec$ to $\mathfrak B(L)$ yields a de Vries algebra $(\mathfrak B(L),\prec)$. Moreover, if $h:L\to M$ is a frame homomorphism between compact regular frames, then $\mathfrak B(h):\mathfrak B(L)\to\mathfrak B(M)$ is a de Vries morphism, where $\mathfrak B(h)(a)=h(a)^{**}$. This defines a covariant functor $\mathfrak B:\KRFrm\to\dev$ which yields an equivalence between $\KRFrm$ and $\dev$: 

\begin{theorem}  \cite{Bez12} \label{thm: KRFrm = DeV}
$\KRFrm$ is equivalent to $\dev$.
\end{theorem}

\begin{definition}
Let $A \in \bal$. For $S\subseteq A$, let $\ann_A(S) = \{ a \in A : as = 0 \ \forall s \in S \}$ be the annihilator of $S$. 
\end{definition}

It is a standard fact of commutative ring theory that $\ann_A(S)$ is an ideal of $A$. As usual, we call an ideal $I$ of $A$ an {\em annihilator ideal} if $I=\ann_A(S)$ for some $S\subseteq A$.  The next lemma can be proved more quickly using Remark~\ref{rem: properties of arch}, but in keeping with our approach we give a choice-free proof that avoids the use of maximal ideals.

\begin{lemma}\label{lem:ann}
Let $A \in \bal$. If $I$ is an annihilator ideal of $A$, then $I$ is an archimedean $\ell$-ideal of $A$.
\end{lemma}

\begin{proof}
Since $I$ is an annihilator ideal, $I = \ann_A(S)$ for some $S \subseteq A$. 
We first show that $I$ is an $\ell$-ideal. Let $a \in A$ and $b \in I$ such that 
 $|a| \le |b|$. Then for each $s \in S$ we have $0 \le |as| \le |bs| = 0$. Therefore, $|as| = 0$, and so $as = 0$. Thus, $a \in I$ and hence $I$ is an $\ell$-ideal.

Since $as=0$ iff $|as| = 0$ iff  $|a|s|| = 0$ iff $a|s|=0$, we have that $\ann_A(S)=\ann_A(\{ |s| : s \in S\})$. Thus, we can assume that $0 \le s$ for each $s \in S$.
To see that $I$ is archimedean, we utilize the following characterization of archimedean $\ell$-ideals \cite[Prop.~4.8]{BCM21d}: An $\ell$-ideal $J$ is archimedean iff $(n|a|-1)^+ \in J$ for each $n \ge 1$ implies $a \in J$. 
Let $a \in A$ with $(n|a|-1)^+ \in I$ for each $n \ge 1$. If $s \in S$, then $(n|a|-1)^+ \cdot s = 0$, so 
\[
0 = (n|a|-1)^+ \cdot s = [(n|a|-1) \vee 0] \cdot s = (n|a|-1)s \vee 0,
\]
where the last equality follows from the $f$-ring identity \cite[Cor.~XVII.6.1]{Bir79} since $s \ge 0$. Therefore, $(n|a|-1)s \le 0$, and hence $n|a|s \le s$ 
for each $n \ge 1$. Since $A$ is archimedean, $|a|s \le 0$. Because $s \ge 0$, this forces $|a|s = 0$, so $as = 0$. Thus, $a \in I$, and so $I$ is an archimedean $\ell$-ideal.
\end{proof}

\begin{definition}
For $A\in\bal$, let $\ann(A)$ be the set of annihilator ideals of $A$, ordered by inclusion. 
\end{definition}

By Lemma~\ref{lem:ann}, $\ann(A)$ is a subposet of $\arch(A)$. 
We next show that $\ann(A)$ is the booleanization of $\arch(A)$. 

\begin{proposition}\label{prop:bool}
For $A\in\bal$, the booleanization of $\arch(A)$ is $\ann(A)$.
\end{proposition}

\begin{proof}

We first show that $I^* = \ann_A(I)$ for each $I\in\arch(A)$.
Since $A$ has no nonzero nilpotent elements \cite[p.~63, Cor.~3]{BP56}, $I \cap \ann_A(I) = 0$, so $\ann_A(I) \subseteq I^*$. Conversely, $II^* \subseteq I \cap I^* = 0$, so $I^* \subseteq \ann_A(I)$. Therefore, $I^* = \ann_A(I)$. From this it follows that 
\[
I\in\mathfrak B(\arch(A)) \mbox{ iff } I=I^{**} \mbox{ iff } I=\ann_A(\ann_A(I)) \mbox{ iff } I\in\ann(A).
\] 
Thus, $\mathfrak B(\arch(A))=\ann(A)$.
\end{proof}

\begin{remark}
The proof that $A$ has no nonzero nilpotent elements given in \cite[p.~63]{BP56} uses the fact that every $f$-ring embeds in a product of linearly ordered $f$-rings, which requires the axiom of choice. In Remark~\ref{rem: no nilpotents} we give an alternate choice-free proof of the fact that $A \in \bal$ has no nonzero nilpotent elements. 
\end{remark}

By Proposition~\ref{prop:bool}, $\ann(A)$ is a complete boolean algebra, and
so as discussed before Theorem~\ref{thm: KRFrm = DeV},
 $(\ann(A),\prec)$ is a de Vries algebra,  where $\prec$ is the restriction of the well-inside relation on $\arch(A)$ given by $I \prec J$ if $I^* \vee J = A$. Moreover, combining Propositions~\ref{thm: arch frame} and \ref{prop:bool}
yields the following:

\begin{theorem}
$\ann : \bal \to \dev$ is a covariant functor, and the following diagram commutes. 
\[
\begin{tikzcd}
& \bal \arrow[dl, "\arch"'] \arrow[dr, "\ann"] & \\
\KRFrm \arrow[rr, "\mathfrak B"'] && \dev
\end{tikzcd}
\]
\end{theorem}

\begin{remark} \label{rem: ann = RO}
Let $A \in \bal$. It is known (see, e.g., \cite[Rem.~4.5]{BCM21d}) that $\arch(A)$ is isomorphic to the frame $\mathcal{O}(Y(A))$ of opens of the Yosida space $Y(A)$.
Since the booleanization of $\mathcal{O}(Y(A))$ is $\RO(Y(A))$, we obtain that $\ann(A)$ is isomorphic to $\RO(Y(A))$ by Proposition~\ref{prop:bool}. 
\end{remark}

\section{Dedekind completions, proximities, and the Dieudonn\'{e} lemma}\label{sec: Dieudonne}

As we saw in the previous section, we have a covariant functor $\ann:\bal\to\dev$. It is less obvious how to construct a covariant functor $\dev\to\bal$. 
Using de Vries and Gelfand dualities, if $X\in\KHaus$, then the corresponding de Vries and $\bal$-algebras are $\RO(X)$ and $C(X)$. 
As we will see shortly, there is an ambient algebra that contains both $C(X)$ and $\RO(X)$. We can then define a proximity on this algebra that will allow us to recover both $C(X)$ and $\RO(X)$. 
This approach is based on Dedekind completions and Dilworth's theorem discussed in the introduction.

We recall that $A\in\bal$ is a {\em Dedekind algebra} if each nonempty subset of $A$ bounded above has a sup, and hence each nonempty subset of $A$ bounded below has an inf. Let $\dbal$ be the full subcategory of $\bal$ consisting of Dedekind algebras. As was pointed out in \cite[Rem.~3.5]{BMO13b}, $\dbal$ is in fact a full subcategory of $\ubal$. 

A {\em Dedekind completion} of $A\in\bal$ is a pair $(D(A),\delta_A)$, where $D(A)$ is a Dedekind algebra and $\delta_A:A\to D(A)$ is a $\bal$-monomorphism such that the image is join-dense (and hence meet-dense) in $D(A)$. It follows from the work of Nakano \cite{Nak50} and Johnson \cite{Joh65} that Dedekind completions exist in $\bal$:

\begin{theorem} \cite[Thm.~3.1]{BMO13b}
For each $A\in\bal$ there exists a unique up to isomorphism Dedekind algebra $D(A)$ and a $\bal$-monomorphism $\delta_A:A\to D(A)$ such that $\delta_A[A]$ is join-dense $($and hence meet-dense$)$ in $D(A)$.
\end{theorem}

By Dilworth's theorem mentioned in the introduction, $D(A)$ is isomorphic to the algebra $N(Y(A))$ of normal functions on the Yosida space of $A$:

\begin{theorem} \cite[Prop.~4.7 and Rem.~4.9]{BMO16}
If $A\in\bal$, then up to isomorphism, the pair $(N(Y(A)),\zeta_A)$ is the Dedekind completion of $A$.
\end{theorem}

\begin{remark} \label{rem: N(X)}
Let $X$ be a topological space. Recalling from the introduction the Baire operators $(-)^*$ and $(-)_*$ on the $\ell$-algebra $B(X)$, 
we have 
\[
N(X) = \{ f \in B(X) : f = (f^*)_* \}.
\] 
Thus, $N(X)$ is not an $\ell$-subalgebra of $B(X)$ since its operations are not pointwise while those of $B(X)$ are. In fact, 
the operations on $N(X)$ are ``normalizations'' of the pointwise operations on $B(X)$ (see, e.g., \cite{Dan15,BMO16}). For example, if $+$ is the pointwise addition, then its normalization is 
\[
f \oplus g = ((f+g)^*)_*.
\] 
The other operations on $N(X)$ are defined similarly using normalization. (However, unlike join, meet in $N(X)$ is pointwise.)
\end{remark}

\begin{notation}
To simplify notation, we identify $A \in \bal$ with its image $\delta_A[A]$ in $D(A)$ and view $\delta_A$ as an inclusion map.
\end{notation}

Let $X\in\KHaus$. It is easy to see that $C(X)$ is a $\bal$-subalgebra of $N(X)$. Moreover, if $U\in\RO(X)$, then the characteristic function $\chi_U$ of $U$ is a normal function, and we can identify $\RO(X)$ with the idempotents of $N(X)$ (see, e.g., \cite[Lem~6.5]{BCM21d}). Thus, $N(X)$ is our desired ambient algebra containing both $C(X)$ and $\RO(X)$.

To recover $C(X)$ from $N(X)$, we utilize the Kat{\v{e}}tov-Tong theorem discussed in the introduction, which implies that if $f,g\in N(X)$ with $f^*\le g$, then there is $h\in C(X)$ such that $f\le h\le g$. This allows us to define a proximity relation $\lhd$ on $N(X)$ by setting $f\lhd g$ iff $f^*\le g$. The Kat{\v{e}}tov-Tong theorem then yields that $C(X)$ is exactly the algebra $\{ f \in N(X) : f \lhd f\}$. 
Because of this, we call $\lhd$ the {\em Kat{\v{e}}tov-Tong proximity} or {\em KT-proximity} for short. The pairs $(N(X),\lhd)$, where $\lhd$ is a KT-proximity, were axiomatized in \cite{BMO16} using the following notion of proximity.

\begin{definition} \label{def: proximity bal}
Let $A \in \bal$. We call a binary relation $\lhd$ on $A$ a {\em proximity} if the following axioms are satisfied:
\begin{enumerate}
\item[(P1)] $0\lhd 0$ and $1 \lhd 1$;
\item[(P2)] $a \lhd b$ implies $a \le b$;
\item[(P3)] $a \le b \lhd c \le d$ implies $a \lhd d$;
\item[(P4)] $a \lhd b,c$ implies $a \lhd b \wedge c$;
\item[(P5)] $a \lhd b$ implies $-b \lhd -a$;
\item[(P6)] $a \lhd b$ and $c \lhd d$ imply $a+c \lhd b+d$;
\item[(P7)] $a \lhd b$ and $0 < r\in\mathbb R$ imply $ra \lhd rb$;
\item[(P8)] $a,b,c,d \ge 0$ with $a \lhd b$ and $c \lhd d$ imply $ac \lhd bd$;
\item[(P9)] $a \lhd b$ implies there is $c \in A$ with $a \lhd c \lhd b$;
\item[(P10)] $a > 0$ implies there is $0 < b\in A$ with $b \lhd a$.
\end{enumerate}
We call the pair $(A,\lhd)$ a {\em proximity $\bal$-algebra}.
\end{definition}

\begin{remark} \label{rem: join}
Since $-(a \vee b) = (-a) \wedge (-b)$, it follows from (P4) and (P5) that $a, b \lhd c$ implies $a \vee b \lhd c$. This will be used in the proof of Theorem~\ref{thm: Dieudonne}.
\end{remark}

Let $(A, \lhd)$ be a proximity $\bal$-algebra. We call $a \in A$ \emph{reflexive} if $a \lhd a$. Let $\dv{R}(A, \lhd)$ be the set of reflexive elements of $(A,\lhd)$. It is an easy consequence of the proximity axioms that $\dv{R}(A, \lhd)$ is an $\ell$-subalgebra of $A$ (see \cite[Lem.~8.10(1)]{BMO16}). 

Clearly each $r\in\mathbb R$ is a reflexive element of $(A,\lhd)$, but in general these might be the only reflexive elements of $(A,\lhd)$. Therefore, to make sure that we have lots of reflexive elements, we need to strengthen (P9).

\begin{definition} \label{def: proximity dedekind algebra}
Let $D$ be a Dedekind algebra. We call a proximity $\lhd$ on $D$ a {\em Kat{\v{e}}tov-Tong proximity} or {\em KT-proximity} for short if (P9) is strengthened to
\begin{enumerate}
\item[(KT)] $a \lhd b$ implies there is $c\in \dv{R}(D, \lhd)$ with $a \lhd c \lhd b$.
\end{enumerate}
We call the pair $(D,\lhd)$ a {\em proximity Dedekind algebra}.
\end{definition}

\begin{remark}
Our terminology is slightly different from that in \cite{BMO16}, where proximities on $\bal$-algebras were first introduced. 
\end{remark}

\begin{remark} \label{ex: prec_A}
\hfill
\begin{enumerate}
\item Typical examples of proximity Dedekind algebras can be constructed as follows. Let $A \in \bal$ and $D = D(A)$ be its Dedekind completion. Define $\lhd_A$ on $D$ by
\[
f\lhd_A g \mbox{ iff } \exists a\in A : f\le a \le g.
\] 
It is elementary to check that $\lhd_A$ satisfies all the proximity axioms save (P10), for which we need to recall the notion of an essential subalgebra. 
An $\ell$-subalgebra $A$ of $D \in \dbal$ is \emph{essential} if $A \cap I \ne 0$ for each nonzero $\ell$-ideal $I$ of $D$. By \cite[Prop.~2.12]{BMO16}, $A$ is essential in $D$ iff for each $0 < d \in D$ there is $0 < a \in A$ with $a \le d$. Thus, essentiality of $A$ in $D$ is equivalent to (P10) for $\lhd_A$. 

\item More generally, \cite[Prop.~2.12]{BMO16} implies that if $A$ is an $\ell$-subalgebra of a Dedekind algebra $D$, then $A$ is essential in $D$ iff $D$ is (isomorphic to) the Dedekind completion of $A$. In particular, if $\lhd$ is a KT-proximity on $D$ and $A = \dv{R}(D, \lhd)$, then (P10) implies that $A$ is essential in $D$, and hence $D$ is the Dedekind completion of $A$.
\end{enumerate}
\end{remark}  

\begin{definition} \label{def: proximity morphism}
Let $(D,\lhd)$ and $(D' ,\lhd')$ be proximity $\bal$-algebras. We call a map $\alpha : D \to D'$ a \emph{proximity morphism} provided, for all $a,b,c \in D$ with $c\lhd c$ and $0 < r \in \mathbb{R}$, we have:
\begin{enumerate}
\item[(PM1)] $\alpha(0) = 0$ and $\alpha(1) = 1$;
\item[(PM2)] $\alpha(a \wedge b) = \alpha(a) \wedge \alpha(b)$;
\item[(PM3)] $a \lhd b$ implies $-\alpha(-a) \lhd' \alpha(b)$;
\item[(PM4)] $\alpha(b) = \bigvee\{ \alpha(a) : a \lhd b\}$;
\item[(PM5)] $\alpha(ra) = r\alpha(a)$;
\item[(PM6)] $\alpha(a \vee c) = \alpha(a) \vee \alpha(c)$;
\item[(PM7)] $\alpha(a + c) = \alpha(a) + \alpha(c)$;
\item[(PM8)] $c\ge 0$ implies $\alpha(ca) = \alpha(c)\alpha(a)$.
\end{enumerate}
\end{definition}

As was shown in \cite[Thm.~8.12]{BMO16}, proximity Dedekind algebras with proximity morphisms form a category $\pda$, where the composition $\alpha_2\star\alpha_1$ of proximity morphisms $\alpha_1 : (D_1, \lhd_1) \to (D_2, \lhd_2)$ and $\alpha_2 : (D_2, \lhd_2) \to (D_3, \lhd_3)$ is defined by 
\[ 
(\alpha_2\star\alpha_1)(a) = \bigvee \{ \alpha_2(\alpha_1(x)) : x \lhd_1 a \}.
\]

\begin{definition} \label{def: KT-proximity}
Let $D$ be a Dedekind algebra. We call a KT-proximity $\lhd$ on $D$ {\em closed} if $\lhd$ is a closed subset in the product topology on $D\times D$. 
\end{definition}

Typical examples of closed proximities are obtained by taking the pairs $(N(X),\lhd)$ where $X\in\KHaus$ and $\lhd$ is the KT-proximity on $N(X)$. This motivates the following definition.

\begin{definition} \label{def: category KT}
Let $(D,\lhd)$ be a proximity Dedekind algebra. We call $(D,\lhd)$ a {\em Kat{\v{e}}tov-Tong algebra}, or {\em KT-algebra} for short, if $\lhd$ is a closed proximity. Let $\cpda$ be the full subcategory of $\pda$ consisting of KT-algebras.
\end{definition}

One of the main results of \cite{BMO16} yields a dual adjunction between $\pda$ and $\KHaus$ which restricts to a dual equivalence between $\cpda$ and $\KHaus$. This is achieved through the contravariant functors $N:\KHaus\to\cpda$ and $\End:\pda\to\KHaus$.

\[
\begin{tikzcd}
\cpda \arrow[rr, hookrightarrow] && \pda \arrow[dl, "\End"]  \arrow[ll, bend right = 20, "N \circ \End"'] \\
&  \KHaus \arrow[ul,  "N"] &
\end{tikzcd}
\]

The functor $N$ sends $X \in \KHaus$ to $(N(X), \lhd)$, where $\lhd$ is the KT-proximity on $N(X)$. 
On morphisms, $N$ sends a continuous map $\varphi : X \to Y$ to the proximity morphism $N(\varphi) : N(Y) \to N(X)$ given by $N(\varphi)(f) = ((f \circ \varphi)^*)_*$ for each $f \in N(Y)$.
To describe the functor $\End$, we recall from \cite[Sec.~5]{BMO16} that an $\ell$-ideal $I$ of a proximity Dedekind algebra is \emph{round} if $a \in I$ implies there is $b \in I$ with $|a| \lhd b$, and an \emph{end} is a maximal round $\ell$-ideal. The functor $\End$ then sends $(D, \lhd)$ to the \emph{space of ends} of $(D, \lhd)$, where the definition of the topology on the set of ends is similar to the definition of the Zariski topology on the space of maximal $\ell$-ideals of a $\bal$-algebra. 
On morphisms, $\End$ sends a proximity morphism $\alpha : (D, \lhd) \to (D', \lhd')$ to the continuous map $\End(\alpha) :\End(D', \lhd') \to \End(D, \lhd)$ given by
\[
\End(\alpha)(x) = \{ d \in D : |d| \lhd c \mbox{ for some } c \in  \alpha^{-1}(x)\}
\]
for each $x \in \End(D', \lhd')$.

The obtained duality is reminiscent of Gelfand duality, albeit in the language of proximity Dedekind algebras. Indeed, the categories $\bal$ and $\pda$ are equivalent. The covariant functor $\dv{R} : \pda\to\bal$ associates with each proximity Dedekind algebra $(D,\lhd)$ the $\bal$-algebra $\mathfrak R(D,\lhd)$ of reflexive elements of $(D,\lhd)$,
and with each proximity morphism $\alpha:D\to E$ its restriction to $\mathfrak R(D,\lhd)$. The covariant functor $\dv{D} : \bal\to\pda$ associates with each $A\in\bal$ the proximity Dedekind algebra $(D(A),\lhd_A)$ (see Remark~\ref{ex: prec_A}(1)),
and with each $\bal$-morphism $\alpha:A\to B$ the proximity morphism 
$\dv{D}(\alpha):D(A)\to D(B)$ given by 
\[
\dv{D}(\alpha)(f)=\bigvee\{\alpha(a) : a\in A \mbox{ and } a\le f \}.
\] 

The equivalence of $\bal$ and $\pda$ restricts to an equivalence of $\ubal$ and $\cpda$. Thus, we arrive at the following commutative diagram \cite[p.~1130]{BMO16}:
\[
\begin{tikzcd}[column sep = 5pc]
\bal \arrow[r, leftrightarrow] \arrow[d, bend right = 30] & \pda \arrow[d, bend right = 30] \\
\ubal \arrow[u, hookrightarrow] \arrow[r, leftrightarrow] & \cpda \arrow[u, hookrightarrow]
\end{tikzcd}
\]

\begin{remark}\label{rem:BMO proof}
The equivalence of $\bal$ and $\pda$ is proved in \cite{BMO16} directly, without a passage to $\KHaus$. However, the proof of equivalence of $\ubal$ and $\cpda$ is done by representing each $A\in\ubal$ as $C(X)$ and $D(A)$ as $N(X)$ for some $X\in\KHaus$, and then utilizing the Kat{\v{e}}tov-Tong theorem. Thus, the proof of the equivalence of $\ubal$ and $\cpda$ given in \cite{BMO16} is not choice-free.
\end{remark}

We conclude this section by giving a direct choice-free proof of the equivalence between $\ubal$ and $\cpda$. For this we require the Dieudonn\'{e} lemma in the form given below.

The Dieudonn\'{e} technique originates in \cite{Die44}. It was utilized by several authors in different contexts. See, for example, \cite{Edw66,BS75,Lan76,BMO20c}. The version below is formulated in the language of proximity $\bal$-algebras and is one of our main results.

\begin{theorem} [Dieudonn\'{e}'s lemma] \label{thm: Dieudonne}
Let $(S, \lhd)$ be a proximity $\bal$-algebra and $D$ the Dedekind completion of $S$ such that $S$ is uniformly dense in $D$. Then the closure $\lhd'$ of $\lhd$ in $D \times D$ is a KT-proximity, and hence $(D,\lhd')$ is a KT-algebra. 
\end{theorem}

\begin{proof}
Let $A = \{ a \in D : a \lhd' a \}$ be the set of reflexive elements of $(D, \lhd')$. We first show that for any $ f, g \in D$, we have $f \lhd' g$ iff there is $c \in A$ with $f \le c \le g$.

Suppose that there is $c \in A$ with $f \le c \le g$. Since $c\lhd' c$ and $\lhd'$ is the closure of $\lhd$, there are sequences $\{c_n\}, \{d_n\}$ in $S$, both converging (uniformly) to $c$, such that $c_n \lhd d_n$ for each $n$. Because $S$ is uniformly dense in $D$, there is a sequence $\{a_n\}$ in $S$ converging to $f$. Since $f \le c$, the sequence $\{a_n \wedge c_n\}$ converges to $f \wedge c = f$. Therefore, if we replace $a_n$ by $a_n \wedge c_n$, we may assume that $a_n \le c_n$ for each $n$. Similarly, there is a sequence $\{b_n\}$ in $S$ converging to $g$ with $d_n \le b_n$ for each $n$. Thus, $a_n \le c_n \lhd d_n \le b_n$, so $a_n \lhd b_n$ for each $n$, and hence $f \lhd' g$. 

Conversely, suppose that $f \lhd' g$. Then there are sequences $\{a_n\},\{b_n\}$ in $S$ such that $\{a_n\}$ converges to $f$, $\{b_n\}$ converges to $g$, and $a_n\lhd b_n$ for each $n$. Because the two sequences are bounded, there are $r, s \in \mathbb{R}$ such that $r \le a_n, b_m, f, g \le s$ for each $n, m$. By replacing each term $t$ by $(t-r)/(s-r)$, we may assume $0 \le a_n, b_m, f, g \le 1$ for each $n, m$. 
By a standard analysis argument, there are subsequences $\{a_{n_k}\}$ and $\{b_{n_k}\}$ such that
\[
\|a_{n_k} - a_{n_{k+1}}\|, \|b_{n_k} - b_{n_{k+1}}\| \le \frac{1}{2^k} \hspace{.25in}\textrm{for each } k.
\]
Since $a_{n_k} \lhd b_{n_k}$ for each $k$, and because $a_{n_k} \to f$ and $b_{n_k} \to g$, we may replace the original sequences by these subsequences to assume that $\|a_{n} - a_{n+1}\|, \|b_{n} - b_{n+1}\| \le 1/2^n$ for each $n$. From this we see that $a_{n+1} \le a_n + 1/2^n$ and $b_{n} - 1/2^n \le b_{n+1}$ for each $n$. 

We produce a Cauchy sequence $\{c_n\}$ in $S$ satisfying $a_n \lhd c_n \lhd b_n$ and $c_{n} - 1/2^n \lhd c_{n+1} \lhd c_{n} + 1/2^n$ for each $n$. To start, since $a_1 \lhd b_1$, there is $c_1 \in S$ with $a_1 \lhd c_1 \lhd b_1$. Because $b_1 \le 1$, we have $c_1 \lhd 1$, and so $c_1 \lhd 1 \le c_1 + 1$. Thus, $c_1 \lhd c_1 + 1$, and hence $c_1 - 1/2 \lhd c_1 + 1/2$. Since $a_2 \le a_1 + 1/2 \lhd c_1 + 1/2$ and $c_1 - 1/2 \lhd b_1 - 1/2 \le b_2$, we have $a_2 \vee (c_1 - 1/2) \lhd b_2 \wedge (c_1 + 1/2)$ by (P4) and Remark~\ref{rem: join}. Therefore, there is $c_2$ with $a_2 \vee (c_1 - 1/2) \lhd c_2 \lhd b_2 \wedge (c_1 + 1/2)$. 
 Now suppose that $n \ge 2$ and there are $c_1, \dots, c_n \in S$ such that
\begin{enumerate}
\item[(i)$_n$]$a_m \lhd c_m \lhd b_m$ for each $m \le n$,
\item[(ii)$_n$]$c_m - 1/2^m \lhd c_{m+1} \lhd c_m + 1/2^m$ for each $m < n$,
\item[(iii)$_n$]$c_n - 1/2^n \lhd c_n + 1/2^n$.
\end{enumerate}
We have $a_{n+1} \le a_n + 1/2^n \lhd c_n + 1/2^n$ and $c_n - 1/2^n \lhd b_n - 1/2^n \le b_{n+1}$. Consequently, 
\[
a_{n+1} \vee (c_n - 1/2^n) \lhd b_{n+1} \wedge (c_n + 1/2^n).
\]
Therefore, there is $c_{n+1} \in S$ with
\[
a_{n+1} \vee (c_n - 1/2^n) \lhd c_{n+1} \lhd b_{n+1} \wedge (c_n + 1/2^n).
\]
From this we see that $a_{n+1} \lhd c_{n+1} \lhd b_{n+1}$ and $c_n- 1/2^n \lhd c_{n+1} \lhd c_n + 1/2^n$. Thus, (i)$_{n+1}$ and (iii)$_{n+1}$ are verified, as well as (ii)$_{n}$. By induction, we have produced the desired sequence, and (ii)$_{n}$ shows that $c_{n} - 1/2^n \le c_{n+1} \le c_{n} + 1/2^n$, so $\|c_{n+1} - c_n\| \le 1/2^n$. This yields that $\{c_n\}$ is Cauchy. If $c = \lim c_n$, then $c_n- 1/2^n \lhd c_{n+1}$ for each $n$ implies that $c \lhd' c$ and so $c \in A$. Moreover, $f = \lim a_n \le \lim c_n  \le \lim b_n = g$. Therefore, we have proved that $f \lhd' g$ iff there is $c \in A$ with $f \le c \le g$.

We next show that $D$ is isomorphic to $D(A)$. For this it is enough to observe that $A$ is essential in $D$ (see Remark~\ref{ex: prec_A}(2)).
Since $f \ne 0$ and $D$ is the Dedekind completion of $S$, there is $0 \ne b \in S$ with $0 \le b \le f$. Because $(S,\lhd)$ is a proximity $\bal$-algebra, there is $a \ne 0$ in $S$ with $a \lhd b$. From $a\lhd b$ it follows that $a\lhd' b$. Therefore, by the argument above, there is $c \in A$ with $a \le c \le b$, and hence $a \le c \le f$. Since $a \ne 0$, we have $c \ne 0$. Thus, $A$ is essential in $D$. 
Consequently, $D$ is the Dedekind completion of $A$ and ${\lhd'} = {\lhd_A}$, which implies that $\lhd'$ is a KT-proximity.
\end{proof}

We next utilize Dieudonn\'{e}'s lemma to give a choice-free proof of the equivalence of $\ubal$ and $\cpda$.

\begin{theorem} \label{thm: ubal = KT}
The functors $\dv{D}$ and $\dv{R}$ yield an equivalence between $\ubal$ and $\cpda$.
\end{theorem}

\begin{proof}
As mentioned above, the functors $\dv{D} : \ubal \to \cpda$ and $\dv{R} : \cpda \to \ubal$ act on objects by sending $A$ to $(D(A), \lhd_A)$ and $(D, \lhd)$ to $\dv{R}(D, \lhd)$, respectively. As we pointed out in Remark~\ref{rem:BMO proof}, the proof in \cite[Thm.~6.6]{BMO16} that $\dv{D}$ is well defined on objects passes through $\KHaus$ and uses the Kat{\v{e}}tov-Tong theorem. We give a direct choice-free proof of this result, using Dieudonn\'{e}'s lemma instead.

Let $A\in\ubal$. Then $(D(A),\lhd_A)$ is a proximity Dedekind algebra by Remark~\ref{ex: prec_A}(1). Let $\lhd$ be the closure of $\lhd_A$ in $D(A)$. Applying Theorem~\ref{thm: Dieudonne} to $(D(A),\lhd_A)$ yields that $\lhd$ is a KT-proximity and hence is equal to $\lhd_B$, where $B = \dv{R}(D, \lhd)$. We show that $A=B$. If $a \in A$, then $a \lhd_A a$, so $a \lhd a$. This implies that $A \subseteq B$. To see the reverse inclusion, let $b \in B$. Then $(b, b)$ is an element of $\lhd$, so there is a sequence $\{(b_n, b_n')\}$ in $\lhd_A$ converging to $(b,b)$. Since $b_n \lhd_A b_n'$, there is $a_n \in A$ with $b_n \le a_n \le b_n'$. Therefore, $\{a_n\}$ converges to $b$. Since $A$ is uniformly complete, $b \in A$. This yields $B = A$, so ${\lhd}$ and ${\lhd_A}$ are equal. Thus, $\lhd_A$ is a closed proximity, and hence $\dv{D}$ is well defined on objects. That it is also well defined on morphisms and that $\dv{D}$ and $\dv{R}$ yield an equivalence of $\ubal$ and $\cpda$ follows from \cite[Cor.~6.8]{BMO16}. 
\end{proof}

Consequently, we obtain the following diagram that commutes up to natural isomorphism (see \cite[Lem.~7.3]{BMO16}).
\[
\begin{tikzcd}
\ubal \arrow[rr, shift left = .5ex, "\dv{D}"] \arrow[dr, shift left = .5ex, "Y"] && \cpda \arrow[ll, shift left = .5ex, "\dv{R}"] \arrow[dl, shift right = .5ex, "\End"'] \\
& \KHaus \arrow[ul, shift left = .5ex, "C"] \arrow[ur, shift right = .5ex, "N"']
\end{tikzcd}
\]

\begin{remark}
Dieudonn\'{e}'s lemma can also be used to give a simple description of the functor $N \circ \End : \pda \to \cpda$. Namely, for each $(D, \lhd) \in \pda$, we have that $N(\End(D, \lhd))$ is naturally isomorphic to $(D, \lhd')$, where $\lhd'$ is the closure of $\lhd$ in $D \times D$. This gives a choice-free description of the reflector $N \circ \End : \pda \to \cpda$.
\end{remark}

\section{Proximity Dedekind algebras and de Vries algebras} \label{sec: pda and dev}

As we saw in the previous section, if $X\in\KHaus$, then $N(X)$ is the Dedekind completion of $C(X)$, and $C(X)$ can be recovered from the Kat{\v{e}}tov-Tong algebra $(N(X),\lhd)$ as the algebra of reflexive elements. This led to a direct choice-free proof that $\ubal$ is equivalent to $\cpda$.

We next concentrate on the connection between $\RO(X)$ and $N(X)$. As we already pointed out, $\RO(X)$ can be identified with the boolean algebra $\Id(N(X))$ of idempotents of $N(X)$. As we will see in this section, the de Vries proximity on $\RO(X)$ is the restriction of the KT-proximity on $N(X)$. For this it is convenient to recall from Section~\ref{sec:ann} that $(\RO(X),\prec)$ is isomorphic to the de Vries algebra $(\ann(C(X)), \prec)$. Thus, it is sufficient to prove that $\ann(C(X))$ and $\Id(N(X))$ are isomorphic as boolean algebras and that the isomorphism preserves and reflects $\prec$,
 where $\prec$ on $\Id(N(X))$ is the restriction of the KT-proximity on $N(X)$. From this it will follow that $(\Id(N(X)),\prec)$ is a de Vries algebra. 

We will give a purely algebraic proof of this result by showing that if $(D,\lhd)$ is a proximity Dedekind algebra and $A$ is the $\bal$-algebra of its reflexive elements, then $(\Id(D),\prec)$ is isomorphic to $(\ann(A),\prec)$. This yields that $(\Id(D),\prec)$ is a de Vries algebra. We conclude the section by showing that associating with each proximity Dedekind algebra $(D, \lhd)$ the de Vries algebra $(\Id(D), \prec)$ defines a covariant functor $\Id$ from $\pda$ to $\dev$.

Let $D$ be a Dedekind algebra. Then $D$ is a Baer ring, where we recall that a commutative ring (with 1) is a {\em Baer ring} if each annihilator ideal is generated by a single idempotent. In fact, as was shown in \cite{BMO13b}, $A\in\bal$ is a Dedekind algebra iff $A\in\ubal$ and $A$ is Baer. However, the proof utilized Gelfand duality. For our purposes it is convenient to give a choice-free proof of this result. In this section we prove the left-to-right implication. The right-to-left implication will be proved in Corollary~\ref{cor: Baer ubal is Dedekind}.

\begin{lemma} \label{lem: Dedekind is Baer}
If $D$ is a Dedekind algebra, then $D \in \ubal$ and $D$ is a Baer ring.
\end{lemma}

\begin{proof}
That $D\in\ubal$ is easy to verify; 
see, e.g., \cite[Rem.~3.5]{BMO13b}. To show that $D$ is Baer, let $S \subseteq D$ and set $I = \ann_D(S)$. Because $D$ is Dedekind, $e := \bigvee \{ a \in I : a \le 1\}$ exists in $D$. We show that $e\in\Id(D)$ and that $e$ generates $I$. Let $s \in S$. As we saw in the proof of Lemma~\ref{lem:ann}, $\ann_D(S) = \ann_D(\{ |s| : s \in S\})$. Therefore, we may assume that $0 \le s$. Thus, by \cite[Lem.~1]{Joh65},
\[
se = s\bigvee \{ a : a \in I, a \le 1 \} = \bigvee \{ sa : a \in I, a \le 1 \} = 0,
\]
so $e \in I$. To see that $e$ is an idempotent, by \cite[Lem.~6.3]{BCMO22a} it is sufficient to observe that $e=2e \wedge 1$. We have $e \le 2e \wedge 1$ since $0 \le e \le 1$. Also, $2e \wedge 1\le 2e$ and $2e \in I$ since $e\in I$. Therefore, $2e \wedge 1 \in I$ because $I$ is an annihilator ideal, hence an $\ell$-ideal by Lemma~\ref{lem:ann}. But then $2e \wedge 1\le e$ by the definition of $e$. Thus, $e=2e\wedge 1$, and so $e \in \Id(D)$. It is left to show that $eD=I$. The inclusion $eD \subseteq I$ is clear since $e \in I$. For the reverse inclusion, let $a \in I$. As $D$ is bounded, there is $n$ with $|a| \le n$. Therefore, $|a|/n \le 1$. Since $|a|/n \in I$, by the definition of $e$, we have $|a|/n \le e$, so $|a| \le ne$. Since $eD = \ann_D((1-e)D)$, it is an $\ell$-ideal by Lemma~\ref{lem:ann}. Thus, $a \in eD$, so $I = eD$, and hence $D$ is Baer.
\end{proof}

\begin{remark} \label{rem: complete idempotents}
Let $D$ be a Dedekind algebra. As we just saw, $D$ is a Baer ring, and hence $\Id(D)$ is a complete boolean algebra (see, e.g., \cite[Prop.~1.4.1]{Berb72}). In fact, arguing as in the proof of Lemma~\ref{lem: Dedekind is Baer} gives that if $D \in \bal$ is Baer and $S \subseteq \Id(D)$, then the join of $S$ in $D$ is the join of $S$ in $\Id(D)$.
\end{remark}

For $A \in \bal$ we recall from Section~\ref{sec:ann} that $(\ann(A), \prec)$ is a de Vries algebra, where $I \prec J$ if $\ann_A(I) + J = A$.

\begin{theorem} \label{prop: sigma D an iso}
Let $(D, \lhd) \in \pda$ and $A = \mathfrak R(D,\lhd)$. The map $\sigma_D : \Id(D) \to \ann(A)$ given by $\sigma_D(e) = eD \cap A$ is a well-defined boolean isomorphism such that $e \lhd f$ iff $\sigma_D(e) \prec \sigma_D(f)$ for all $e,f \in \Id(D)$.
\end{theorem}

\begin{proof}
We first show that for each $\ell$-ideal $I$ of $D$, we have $\ann_A(I \cap A) = \ann_D(I) \cap A$.
The inclusion $\supseteq$ is clear. For the reverse inclusion, let $a \in \ann_A(I \cap A)$ and $x \in I$. By Remark~\ref{ex: prec_A}(2), $D$ is the Dedekind completion of $A$, so $A$ is join-dense in $D$, and hence $|x| = \bigvee \{ b \in A : 0 \le b \le |x|\}$. Therefore, $|a||x| = \bigvee \{ |a|b : 0 \le b \le |x|\}$. If $b \in A$ with $0 \le b \le |x|$, then $b \in I \cap A$, so $ab = 0$ as $a \in \ann_A(I \cap A)$. Thus, $|a|b = 0$, so $|a||x| = 0$, and hence $ax = 0$. Consequently, $a \in \ann_D(I) \cap A$.

Let $e \in \Id(D)$. By the previous paragraph,
\[
eD \cap A = \ann_D((1-e)D) \cap A = \ann_A((1-e)D \cap A),
\]
so $eD \cap A\in\ann(A)$, and hence $\sigma_D$ is well defined.

We next show that $\sigma_D$ is an order isomorphism. Let $e,f \in \Id(D)$ with $e \le f$. Then $e = ef$. Suppose $a \in eD \cap A$. We have $a = ea$, so $a = efa = fea$, and so $eD \cap A \subseteq fD \cap A$. Conversely, suppose $eD \cap A \subseteq fD \cap A$. Since $A$ is join-dense in $D$ and $eD \cap A$, $fD \cap A$ are $\ell$-ideals in $A$,
\begin{align*}
e &= \bigvee \{ a \in A : 0 \le a \le e\} = \bigvee\{ a \in eD \cap A : 0 \le a \le e\}, \\
f &= \bigvee \{ a \in A : 0 \le a \le f \} = \bigvee\{ a \in fD \cap A : 0 \le a \le f\}.
\end{align*}
Therefore, $e \le f$. To see that $\sigma_D$ is onto, let $I \in \ann(A)$. Then $I = \ann_A(S) = \ann_D(S) \cap A$ for some $S \subseteq A$.
By Lemma~\ref{lem: Dedekind is Baer}, $D$ is a Baer ring, so there is $e \in \Id(D)$ with $\ann_D(S) = eD$. Thus, $I = eD \cap A$. Consequently, $\sigma_D$ is an order isomorphism, hence a boolean isomorphism.

Finally, to see that $e \lhd f$ iff $\sigma_D(e) \prec \sigma_D(f)$, first suppose that $e \lhd f$. By (KT), there is $a \in A$ with $e \le a \le f$. Therefore, $a \in fD \cap A$. Also, $0 \le 1-a \le 1-e$, so $1-a \in (1-e)D \cap A$. This implies $((1-e)D \cap A) + (fD \cap A) = A$. By the first paragraph, $\ann_A(eD \cap A) = \ann_D(eD) \cap A = (1-e)D \cap A$. Thus, $\ann_A(eD \cap A) + (fD \cap A) = A$, and so $eD \cap A \prec fD \cap A$.

For the converse, suppose that $eD \cap A \prec fD \cap A$. Then $((1-e)D \cap A) + (fD\cap A) = A$. 
By \cite[Lem.~A.2(1)]{BCM21d}, if $I,J$ are $\ell$-ideals of $A$ with $I+J = A$, then there is $a \in I$ with $0 \le a \le 1$ and $1-a \in J$. Therefore, 
there is $a \in fD \cap A$ with $0 \le a \le 1$ and $1-a \in (1-e)D \cap A$. Since $a \in fD$, we have $a = fa$, so $a = fa \le f \cdot 1 = f$ because $a \le 1$. A similar argument shows $1-a \le 1-e$, so $e \le a$. Thus, $e \le a \le f$, and hence $e \lhd f$.
\end{proof}

\begin{remark}
Let $(D,\lhd)$ and $A$ be as in Theorem~\ref{prop: sigma D an iso}. Since $D$ is a Baer ring, we have
$\ann(D) = \{eD : e \in \Id(D)\}$. Consequently, $\ann(D)$ and $\ann(A)$ are isomorphic boolean algebras via the map that sends $eD$ to $eD \cap A$.
\end{remark}

Let $(D,\lhd)\in\pda$. To prove that the restriction of $\lhd$ to $\Id(D)$ is a de Vries proximity, we need the following lemma.

\begin{lemma} \label{lem: M4}
Let $A \in \bal$.
\begin{enumerate}[$(1)$]
\item Let $I$ be an $\ell$-ideal of $A$. If $0 \le a \in I$ and $0 < \varepsilon \in \mathbb{R}$, then $\ann_A((a-\varepsilon)^+) + I = A$.
\item Let $D$ be the Dedekind completion of $A$. If $f \in \Id(D)$ and $a \in A$ with $0 \le a \le f$, then for each $\varepsilon > 0$ there is $e \in \Id(D)$ with $e \lhd f$ and $a \le e + \varepsilon$.
\end{enumerate}
\end{lemma}

\begin{proof}
(1) To show that $\ann_A((a-\varepsilon)^+) + I = A$, it is sufficient to find $b \in \ann_A((a-\varepsilon)^+)$ such that $1-b \in I$. Therefore, we need $b$ such that $b(a-\varepsilon)^+ = 0$ and $1-b \in I$. We show that $b = \varepsilon^{-1}(a-\varepsilon)^-$ is the desired element. We have $(a-\varepsilon)^+(a-\varepsilon)^- = 0$, so $(a-\varepsilon)^+[\varepsilon^{-1}(a-\varepsilon)^-] = 0$. Thus, $b(a-\varepsilon)^+ = 0$. To see that $1-b \in I$, using standard vector lattice identities, 
\begin{align*}
b &= \varepsilon^{-1}(a-\varepsilon)^- = \varepsilon^{-1}((\varepsilon - a) \vee 0) = (1 - \varepsilon^{-1}a) \vee 0 \\
&= 1 + (-\varepsilon^{-1}a \vee -1) = 1 - (\varepsilon^{-1}a \wedge 1).
\end{align*}
Since $a \in I$, we have $\varepsilon^{-1}a \in I$. From $a \ge 0$ it follows that $0 \le \varepsilon^{-1}a \wedge 1 \le \varepsilon^{-1}a$, so $\varepsilon^{-1}a \wedge 1 \in I$. Consequently, $1-b = \varepsilon^{-1}a \wedge 1 \in I$.

(2) Set $I = fD \cap A$, an $\ell$-ideal of $A$. If $0 \le a \le f$, then $a \in I$. By (1), $\ann_A((a-\varepsilon)^+) + I = A$. By Lemma~\ref{lem: Dedekind is Baer}, $D$ is Baer ring, so there is $e \in \Id(D)$ with
\[
\ann_A(\ann_A((a-\varepsilon)^+)) = \ann_D(\ann_A((a-\varepsilon)^+)) \cap A = eD \cap A.
\]
Thus, $\ann_A(eD\cap A) = \ann_A((a-\varepsilon)^+)$, and hence $\ann_A(eD \cap A) + I = A$. This means that $(eD \cap A) \prec (fD \cap A)$, so $e \lhd f$ by Theorem~\ref{prop: sigma D an iso}. Moreover, $(a-\varepsilon)^+ \in \ann_A(\ann_A((a-\varepsilon)^+)) = eD \cap A$, so $(a - \varepsilon)^+ e = (a - \varepsilon)^+$. Because $0 \le a \le f \le 1$ and $\varepsilon > 0$, we have $(a - \varepsilon)^+ \le a \le 1$, which together with $(a - \varepsilon)^+ e = (a - \varepsilon)^+$ yields $(a - \varepsilon)^+ \le e$. Thus, $a - \varepsilon \le (a - \varepsilon)^+ \le e$, so $a \le e + \varepsilon$.
\end{proof}

We are ready to prove the main result of this section.

\begin{theorem} \label{thm: prox}
\hfill
\begin{enumerate}[$(1)$]
\item Let $(D, \lhd) \in \pda$ and $A = \mathfrak R(D,\lhd)$. Then the restriction $\prec$ of $\lhd$ to $\Id(D)$ is a de Vries proximity on $\Id(D)$ and  $\sigma_D : \Id(D) \to \ann(A)$ is a de Vries isomorphism.
\item There is a covariant functor $\Id : \pda \to \dev$ which sends $(D, \lhd)$ to $(\Id(D), \prec)$ and a proximity morphism $\alpha : (D, \lhd) \to (D', \lhd')$ to its restriction $\alpha|_{\Id(D)}$.
\end{enumerate}
\end{theorem}

\begin{proof}
(1) By Theorem~\ref{prop: sigma D an iso}, $\sigma_D:\Id(D)\to\ann(A)$ is a boolean isomorphism and $e \prec f$ iff $\sigma_D(e) \prec \sigma_D(f)$. Since $(\ann(A), \prec)$ is a de Vries algebra, it follows that $(\Id(D), \prec)$ is a de Vries algebra and $\sigma_D$ is a de Vries isomorphism.

(2) By (1), $(\Id(D), \prec) \in \dev$. We first show that $\alpha$ sends idempotents to idempotents. Let $e \in \Id(D)$. Then $e = 1 \wedge 2e$, so $\alpha(e) = \alpha(1 \wedge 2e) = \alpha(1) \wedge \alpha(2e) = 1 \wedge 2\alpha(e)$. Therefore, $\alpha(e) \in \Id(D')$, so $\alpha|_{\Id(D)}$ is a well-defined map from $\Id(D)$ to $\Id(D')$. We next show that $\gamma:=\alpha|_{\Id(D)}$ is a de Vries morphism. The first two axioms of a de Vries morphism hold for $\gamma$ since they hold for $\alpha$. For (M3), suppose that $e \prec f$. Then $-\alpha(-e) \prec \alpha(f)$. But $-\alpha(-e) = \alpha(e^*)^*$ because 
\begin{equation}
\alpha(e^*)^* = 1 - \alpha(1-e) = 1 - (1 + \alpha(-e)) = -\alpha(-e). \tag{5.1} \label{5.1}
\end{equation}
Therefore, $\gamma(e^*)^* \prec \gamma(f)$. Finally, to show (M4), let $f \in \Id(D)$. Then $\gamma(f) = \alpha(f) = \bigvee \{ \alpha(g) : g \in D, g \prec f\}$. By (KT)
we have $\alpha(f) = \bigvee \{ \alpha(a) : a \in \dv{R}(D, \lhd),  a \le f\}$. Let $0 < \varepsilon \in \mathbb{R}$. Since $D$ is the Dedekind completion of $\dv{R}(D, \lhd)$ by Remark~\ref{ex: prec_A}(2), it follows from Lemma~\ref{lem: M4}(2) that for each $a \in \dv{R}(D)$ with $a \le f$ there is $e \in \Id(D)$ with $e \prec f$ and $a \le e + \varepsilon$. Thus, $\alpha(a) \le \alpha(e) + \varepsilon = \gamma(e) + \varepsilon$, and so
\begin{align*}
\gamma(f) &= \bigvee \{ \alpha(a) : a \in A, a \le f\} \le \bigvee \{ \gamma(e) + \varepsilon : e \in \Id(D), e \prec f \} \\
&= \bigvee \{ \gamma(e) : e \in \Id(D), e \prec f \} + \varepsilon\le \gamma(f) + \varepsilon.
\end{align*}
Since this is true for all $\varepsilon$, we get $\gamma(f) = \bigvee \{ \gamma(e)  : e \in \Id(D), e \prec f \}$
and so (M4) holds. Consequently, $\gamma$ is a de Vries morphism, and we set $\Id(\alpha) = \alpha|_{\Id(D)}$.

It is clear that if $\alpha$ is an identity proximity morphism, then $\Id(\alpha)$ is an identity de Vries morphism. It is left to show that $\Id$ preserves composition. Let $\alpha_1 : (D_1, \lhd_1) \to (D_2, \lhd_2)$ and $\alpha_2 : (D_2, \lhd_2) \to (D_3, \lhd_3)$ be proximity morphisms, and let $\gamma_i = \alpha_i|_{\Id(D_i)}$. If $f \in \Id(D_1)$, then
\[
(\gamma_2 \star \gamma_1)(f) = \bigvee \{ \gamma_2\gamma_1(e) : e \in \Id(D_1), e \prec_1 f\}
\]
and
\[
(\alpha_2\star \alpha_1)(f) = \bigvee \{ \alpha_2\alpha_1(a) : a \in \dv{R}(D_1, \lhd_1), a \le f\}.
\]
Since $e \prec_1 f$ implies that there is $a \in \dv{R}(D_1, \lhd_1)$ with $e\le a\le f$, it follows that $(\gamma_2 \star \gamma_1)(f) \le (\alpha_2\star \alpha_1)(f)$. For the reverse inequality, if $a\le f$ and $\varepsilon>0$, then as above, there is $e\in\Id(D_1)$ with $e\prec_1 f$ and $a \le e + \varepsilon$. Therefore, $\alpha_2\alpha_1(a) \le \alpha_2\alpha_1(e) + \varepsilon$, and since this holds for all $\varepsilon$, we conclude that $(\alpha_2\star \alpha_1)(f) \le (\gamma_2 \star \gamma_1)(f)$. Hence, $(\gamma_2 \star \gamma_1)(f) = (\alpha_2\star \alpha_1)(f)$ for each $f\in\Id(D_1)$.
Thus, $(\alpha_2 \star \alpha_1)|_{\Id(D_1)} = \alpha_2|_{\Id(D_2)} \star \alpha_1|_{\Id(D_1)}$, and so $\Id$ is a covariant functor.
\end{proof}

\section{De Vries algebras and proximity Baer-Specker algebras}\label{sec: proximity Baer-Specker}

To define a functor from $\dev$ to $\cpda$, we need to introduce the concept of Baer-Specker algebras. The same way we can think of de Vries algebras as the algebras $\RO(X)$, where $X$ is compact Hausdorff, and of KT-algebras as the algebras $N(X)$, we can think of Baer-Specker algebras as the algebras $FN(X)$ of finitely-valued normal functions. These algebras have a long history, for which we refer to \cite{BMO20e} and the references therein. In our context they arise as follows. 

\begin{definition} \cite[Def.~5.1]{BMO13a}
We call a commutative unital $\mathbb{R}$-algebra $A$ a \emph{Specker algebra} if it is generated as an $\mathbb{R}$-algebra by its idempotents.
\end{definition}

 For $A\in\bal$ let $\spec(A)$ be the $\mathbb{R}$-subalgebra of $A$ generated by $\Id(A)$. We call $\spec(A)$ the \emph{Specker subalgebra} of $A$.

\begin{theorem} \cite[Prop.~5.5]{BMO13a}
Each Specker algebra is a $\bal$-algebra. Thus, $A \in \bal$ is a Specker algebra iff $A = \spec(A)$.
\end{theorem}

\begin{definition}
A Specker algebra is a \emph{Baer-Specker} algebra if it is a Baer ring.
\end{definition}

\begin{remark} \label{rem: Baer vs idempotents}
By \cite[Cor.~4.4]{BMMO15a}, a Specker algebra $A$ is Baer-Specker iff $\Id(A)$ is a complete boolean algebra. We will use this fact frequently. 
\end{remark}

It is proved in \cite[Thm.~6.2]{BMO13a} that $A\in\bal$ is a Specker algebra iff $A$ is isomorphic to the $\ell$-algebra $FC(X)$ of finitely-valued continuous functions on a Stone space $X$, and that the category of Specker algebras is dually equivalent to the category of Stone spaces. Moreover, $A$ is a Baer-Specker algebra iff $A$ is isomorphic to $FC(X)$ where $X$ is in addition extremally disconnected (ED), and the category of Baer-Specker algebras is dually equivalent to the category of compact Hausdorff ED-spaces.

Proximities on Specker algebras and Baer-Specker algebras were introduced in \cite{BMMO15b}, where it was shown that the category of proximity Baer-Specker algebras is equivalent to $\dev$ and dually equivalent to $\KHaus$. 

\begin{definition}\label{def:PBSp} 
We call a proximity $\bal$-algebra  $(A,\lhd)$ a {\em proximity Specker algebra} if $A$ is a Specker algebra. If $A$ is a Baer-Specker algebra, then we call $(A,\lhd)$ a {\em proximity Baer-Specker algebra}.
\end{definition}

\begin{remark}
In \cite[Def.~4.2]{BMMO15b} the base ring is an arbitrary totally ordered integral domain $R$ rather than $\mathbb{R}$. Because of this, Axiom (P7) takes on the following more complicated form: 
\[
a \lhd b \textrm{ implies } ra \lhd rb \textrm{ for each } 0 < r \in R, \textrm{ and } ra \lhd rb \textrm{ for some } 0 < r \in R \textrm{ implies } a \lhd b. 
\]
If $R$ is a totally ordered field, this axiom simplifies to (P7) of Definition~\ref{def: proximity bal}.
\end{remark}

To distinguish between proximity $\bal$-algebras and proximity Specker algebras, from now on we will write $S$ for a Specker algebra and $\ll$ for a proximity on $S$.

\begin{definition} \label{def: weak proximity morphism}
Let $(S,\ll)$ and $(S',\ll')$ be two proximity Baer-Specker algebras. A map $\alpha:S\to S'$ is a {\em weak proximity morphism} if $\alpha$ satisfies axioms (PM1)--(PM5) of Definition~\ref{def: proximity morphism}, as well as the following weakening of axioms (PM6) and (PM7), where $r\in\mathbb R$: 
\begin{enumerate}
\item[(PM6$^\prime$)] $\alpha(a \vee r) = \alpha(a) \vee r$.
\item[(PM7$^\prime$)] $\alpha(a + r) = \alpha(a) + r$.
\end{enumerate}
\end{definition}

\begin{remark}\label{rem: wpm=pm}
\hfill
\begin{enumerate}
\item The weakening of (PM8) is (PM5), hence it is redundant.
\item Definition~\ref{def: weak proximity morphism} originates in \cite[Def.~6.4]{BMMO15b}, where morphisms between proximity Baer-Specker algebras were called proximity morphisms. Here we call them weak proximity morphisms because this notion is weaker than that of a proximity morphism given in Definition~\ref{def: proximity morphism}. However, as we will see in Theorem~\ref{prop:weak proximity morphism}, the two notions of morphism between proximity Baer-Specker algebras are equivalent. This requires several technical lemmas, which are proved in the Appendix.
\item If $\alpha$ is a proximity morphism between proximity Dedekind algebras, then it is obvious that in Axiom (PM4) the least upper bound of $\{\alpha(a):a\lhd b\}$ exists. That this least upper bound also exists if $\alpha$ is a weak proximity morphism between proximity Baer-Specker algebras follows from \cite[Lem.~6.1]{BMMO15b}. 
\end{enumerate}
\end{remark}

\begin{theorem} \cite[Thm.~6.7]{BCMO22a}
Proximity Baer-Specker algebras and weak proximity morphisms between them form a category $\PBSp$, where the composition $\alpha_2\star\alpha_1$ of two proximity morphisms $\alpha_1:S_1\to S_2$ and $\alpha_2:S_2\to S_3$ is 
given by
\[
(\alpha_2\star\alpha_1)(s)=\bigvee\{\alpha_2\alpha_1(t) : t\ll_1 s\}.
\]
\end{theorem}

That $\PBSp$ is equivalent to $\dev$ was first observed in \cite[Cor.~8.7]{BMMO15b}, but the proof used duality theory for these categories. A purely algebraic and choice-free proof of this result was given in \cite[Thm.~6.9]{BCMO22a}:

\begin{theorem} \label{thm: dev = PBSp}
The categories $\PBSp$ and $\dev$ are equivalent.
\end{theorem}

We recall that this equivalence is obtained as follows. The covariant functor $\Id : \PBSp \to \dev$ sends a proximity Baer-Specker algebra ${(S, \ll)}$ to the de Vries algebra $(\Id(S), \prec)$, where $\prec$ is the restriction of $\ll$ to $\Id(S)$, and a proximity morphism $\alpha : (S, \ll) \to (S', \ll')$ to its restriction $\alpha|_{\Id(S)}$. 
The covariant functor from $\dev$ to $\PBSp$ is defined by generalizing the notion of a boolean power of $\mathbb R$ to that of a de Vries power. 

\begin{definition} \label{star definition} \cite[Def.~4.7]{BCMO22a}
For a boolean algebra $B$, define $\mathbb R[B]^\flat$ to be the set of all decreasing functions $a:\mathbb R\to B$ for which there exist $1=e_0 > e_1 > \cdots > e_n > 0$ in  $B$ and $r_0 < r_1 < \cdots < r_n$ in $\mathbb R$ such that
\[
a(r) =
\left\{\begin{array}{ll}
1 & \textrm{if } r \le r_0, \\
e_i & \textrm{if } r_{i-1} < r \le r_i, \\
0 & \textrm{if } r_n < r.\\
\end{array}\right.
\]
\end{definition}

By \cite[Thm.~4.9]{BCMO22a}, $\mathbb{R}[B]^\flat$ is a Specker algebra with pointwise order and algebra operations given by 
\begin{itemize}
\item $(a+b)(r) = \displaystyle{\bigvee} \{a(r_1) \wedge b(r_2) : r_1+r_2 \ge r\}$.
\item If $s>0$, then $(sa)(r) = \displaystyle{\bigvee} \{a(t) : st\ge r\}$.
\item If $a,b \geq 0$, then $(ab)(r) =\displaystyle{\bigvee} \{a(r_1) \wedge b(r_2) : r_1, r_2 \ge 0, r_1r_2 \ge r\}$.
\end{itemize}
Moreover, $B$ is isomorphic to $\Id(\mathbb{R}[B]^\flat)$, and the isomorphism sends each $e\in B$ to $e^\flat\in\mathbb R[B]^\flat$ defined by
\[
e^\flat (r) = \left\{ 
\begin{array}{ll}
1 & \textrm{if }r \le 0 \\ 
e & \textrm{if }0 < r \le 1 \\ 
0 & \textrm{if }1 < r. 
\end{array} \right.
\]
Furthermore, each de Vries proximity $\prec$ on $B$ lifts to a proximity $\prec^\flat$ on $\mathbb R[B]^\flat$ given by
\[
a\prec^\flat b \mbox{ iff } a(r)\prec b(r) \mbox{ for all } r\in\mathbb R.
\]
Then for $e, f\in B$ we have
\begin{equation}
e \prec f \textrm{ iff } e^\flat \prec^\flat f^\flat. \tag{6.1} \label{eqn}
\end{equation}
Thus, if $(B,\prec)$ is a de Vries algebra, then $(\mathbb R[B]^\flat,\prec^\flat)$ is a proximity Baer-Specker algebra and $(B,\prec)$ is isomorphic to $(\Id(\mathbb R[B]^\flat),\prec^\flat)$.

In addition, each de Vries morphism $\sigma : (B, \prec) \to (B', \prec')$ extends to $\sigma^\flat : \mathbb{R}[B]^\flat \to \mathbb{R}[B']^\flat$ given by $\sigma^\flat(a) = \sigma \circ a$. By \cite[Thm.~6.5]{BCMO22a}, $\sigma^\flat$ is a weak proximity morphism. The correspondence $B \mapsto B^\flat$ and $\sigma \mapsto \sigma^\flat$ defines a covariant functor $\Sp : \dev \to \PBSp$.
We thus have that $\Id:\PBSp\to\dev$ and $\Sp:\dev\to\PBSp$ are well-defined covariant functors that establish an equivalence of $\PBSp$ and $\dev$. Combining this with de Vries duality and the dual equivalence between $\PBSp$ and $\KHaus$ (see \cite[Thm.~8.6]{BMMO15b}), we obtain the following commutative diagram, where the horizontal arrow is an equivalence, while the slanted arrows are dual equivalences:
\[
\begin{tikzcd}
\PBSp \arrow[rr, leftrightarrow] \arrow[dr, leftrightarrow] && \dev \arrow[dl, leftrightarrow] \\
& \KHaus   &
\end{tikzcd}
\]

\begin{remark} \label{rem: FN = Sp Ro}
Let $X \in \KHaus$. By de Vries duality, $(\RO(X), \prec) \in \dev$. Also, by \cite[Thm.~4.10]{BMMO15b}, $(FN(X), \ll) \in \PBSp$, where we recall from \cite[Def.~3.3]{BMMO15b} that 
\[
f \ll g \textrm{ if }\cl(f^{-1}[r,\infty)) \subseteq g^{-1}[r, \infty) \textrm{ for each } r \in \mathbb{R}.
\]
By \cite[Lem.~4.8]{BMMO15b}, sending $U$ to its characteristic function $\chi_U$ is a boolean isomorphism from $\RO(X)$ to $\Id(FN(X))$. It easily follows from the definitions of $\prec$ and $\ll$ that $U \prec V$ iff $\chi_U \ll \chi_V$. Thus, ${(FN(X), \ll)} \cong \Sp(\RO(X), \prec)$ by Theorem~\ref{thm: dev = PBSp}. In Remark~\ref{rem: S is a functor} we will see that $\ll$ is the restriction to $FN(X)$ of the KT-proximity $\lhd$ on $N(X)$ defined in Section~\ref{sec:ann}.
\end{remark}

\section{Proximity Baer-Specker algebras and Kat{\v{e}}tov-Tong algebras}\label{sec:BS and KT}

In this section we prove that the category $\PBSp$ of proximity Baer-Specker algebras is equivalent to the category $\cpda$ of Kat{\v{e}}tov-Tong algebras, thus completing a series of equivalences and dual equivalences discussed in this paper. 

We can compose the functors $\Id:\pda\to\dev$ and $\Sp : \dev\to\PBSp$ to obtain a covariant functor from $\pda$ to $\PBSp$.

\begin{proposition}
There is a covariant functor $\Sp \circ \Id : \pda \to \PBSp$.
\end{proposition}

\begin{remark}
As we will see in Remark~\ref{rem: S is a functor}, the composition $\Sp\circ\Id$ is naturally isomorphic to the functor that associates to each proximity Dedekind algebra $(D,\lhd)$ the pair $(\spec(D),{\lhd}|_{\spec(D)})$, where we recall that $\spec(D)$ is the Specker subalgebra of $D$. 
\end{remark}

To define a covariant functor $\PBSp\to\pda$ requires some preparation. Let $S$ be a Specker algebra and $B=\Id(S)$. We recall (see \cite[Lem.~2.1]{BMMO15a}) that each $s\in S$ has an \emph{orthogonal decomposition} $s = \sum_{i=0}^n r_i e_i$ with $r_i \in \mathbb{R}$ (not necessarily distinct) and $e_i \in B$ pairwise orthogonal (that is, $e_i \wedge e_j = 0$ for each $i \ne j$). If, in addition, $e_0 \vee \cdots \vee e_n = 1$, we call this a \emph{full orthogonal decomposition}.  

\begin{lemma} \label{lem: S is Sp(D)}
Let $S$ be a Baer-Specker algebra and $D$ its Dedekind completion. Then $S = \spec(D)$.
\end{lemma}

\begin{proof}
It is sufficient to show that $\Id(S)=\Id(D)$. Since $\Id(S)\subseteq\Id(D)$, it then suffices to show the other inclusion. 
Let $e \in \Id(D)$. Since $S$ is join-dense in $D$, we may write $e = \bigvee \{ a \in S : a \le e \}$. Moreover, since $0 \le e$, we have $e = \bigvee \{ a \in S : 0 \le a \le e \}$.  Let $a \in S$ with $0 \le a$. Then $a = \sum_i r_i e_i$ for some $r_i \in \mathbb{R}$ and pairwise orthogonal nonzero idempotents $e_i \in \Id(S)$. Therefore, $ae_i = r_i e_i$ because $e_ie_j = e_i \wedge e_j = 0$ when $i \ne j$. If $r_i < 0$, then $r_i e_i \le 0$, which implies that $r_i e_i = 0$ since $ae_i \ge 0$. This forces $r_i = 0$, a contradiction. Thus, each $r_i \ge 0$. Then $a = \bigvee_i r_i e_i$ by \cite[Eqn.~XIII.3(14)]{Bir79}. Consequently, $a$ is a finite join of elements of the form $rf$ with $0 \le r \in \mathbb{R}$ and $f \in \Id(S)$. Therefore, $e = \bigvee \{ rf : 0 \le r, f \in \Id(S), rf \le e \}$. If $rf \le e$, then $r \le 1$ and $f \le e$ by \cite[Lem.~4.9(6)]{BMMO15b}. Thus, $e = \bigvee \{ f \in \Id(S) : f \le e\}$. Since $S$ is Baer, $\Id(S)$ is a complete boolean algebra by Remark~\ref{rem: Baer vs idempotents}, so this join exists in $\Id(S)$, and is equal to the join in $S$ by Remark~\ref{rem: complete idempotents}. Finally, because $D$ is the Dedekind completion of $S$, an existing join in $S$ is the same as the corresponding join in $D$, and hence $e \in \Id(S)$.
\end{proof}

\begin{proposition} \label{prop: S dense in A}
Let $A \in \bal$ be Baer. Then $\spec(A)$ is uniformly dense in $A$.
\end{proposition}

\begin{proof}
Let $a \in A$. We claim that it is enough to show that whenever $0 \le a \in A$, there is $b \in S:=\spec(A)$ with $b \le a \le b + 1$. Suppose this happens. We show that $S$ is uniformly dense in $A$. Let $a \in A$ be arbitrary. There is $r \in \mathbb{R}$ with $a + r \ge 0$. Given $\varepsilon > 0$ there is $n \in \mathbb{N}$ with $1/n < \varepsilon$. By assumption there is $b \in S$ with $b \le n(a + r) \le b + 1$. Therefore, $b/n - r \le a \le 
b/n + 1/n - r$. Set $c = b/n - r$. Then $c \in S$ and $c \le a \le c + 1/n$. This implies that $\|a - c\| \le 1/n < \varepsilon$. 
Thus, $S$ is unformly dense in $A$.

We now show that if $0 \le a \in A$, there is $b \in S$ with $b \le a \le b + 1$. For each $n \ge 1$ set $a_n = (a \wedge n) - (a \wedge (n-1))$. Then $a_n = [(a - (n-1)) \wedge 1] \vee 0$ by \cite[Lem.~5.4(1)]{BCMO22a}\footnote{The hypothesis of the lemma has $S$ a Specker algebra but the proof of (1) does not use that.}. There is a positive integer $N$ with $a \le N$. This implies that $a_n = 0$ if $n > N$, and so
\begin{align*}
a &= (a \wedge 1) + [(a \wedge 2) - (a \wedge 1)] + \cdots + [(a\wedge N) - (a \wedge (N-1))]\\
&= a_1 + \cdots + a_{N}.
\end{align*}
We will show that there are idempotents $e_n \in S$ satisfying $a_{n+1} \le e_n \le a_n$ for each $n$ with $1 \le n \le N$. From this,
setting $b = e_1 + \cdots + e_{N}$, we obtain $b \le a_1 + \cdots + a_{N} = a$. Also, since $a_1 \le 1$, we have $a \le 1 + e_1 + \cdots + e_{N-1} \le 1 + b$.

To produce the idempotents, since $A$ is Baer, there is $e_n \in \Id(A) = \Id(S)$ with $e_nA = \ann_A((a-n)^-)$. Since $(a-n)^+ (a-n)^- = 0$, we have $(a-n)^+ \in e_n A$, so
\[
a_{n+1} =[(a-n) \wedge 1] \vee 0 = [(a-n) \vee 0] \wedge 1 = (a-n)^+ \wedge 1
\]
by \cite[Thm.~XIII.4.4]{Bir79}.
Therefore, $a_{n+1} \in e_n A$ because $e_n A$ is an $\ell$-ideal of $A$ by Lemma~\ref{lem:ann}. Thus, $a_{n+1}e_n = a_{n+1}$. This yields $a_{n+1} = a_{n+1}e_n \le e_n$ because $a_{n+1} \le 1$. For the other inequality, since $e_n(a-n)^- = 0$, we have $e_n(a-n) = e_n(a-n)^+ - e_n(a-n)^- = e_n(a-n)^+ \ge 0$. Therefore, by \cite[Cor.~XVII.5.1]{Bir79},
\begin{align*}
e_na_n &= e_n([(a-(n-1)) \wedge 1] \vee 0) = [e_n(a-(n-1)) \wedge e_n] \vee 0 \\
&= [(e_n(a-n) + e_n) \wedge e_n] \vee 0 = e_n
\end{align*}
because $e_n(a-n) + e_n \ge e_n$ (as $e_n(a-n) \ge 0$) and $0 \le e_n$. 
Since $e_n \le 1$, we get $e_n = e_na_n \le a_n$, which gives the other inequality. We have thus produced idempotents $e_n$ with $a_{n+1} \le e_n \le a_n$ for each $n$. This completes the proof.
\end{proof}

\begin{corollary} \label{cor: S dense in D}
A Baer-Specker algebra is uniformly dense in its Dedekind completion.
\end{corollary}

\begin{proof}
Let $S$ be Baer-Specker and $D$ its Dedekind completion. By  Lemma~\ref{lem: S is Sp(D)}, $S = \spec(D)$. Since $D$ is Baer by Lemma~\ref{lem: Dedekind is Baer}, $S$ is uniformly dense in $D$ by Proposition~\ref{prop: S dense in A}.
\end{proof}

\begin{corollary} \label{cor: S(D) essential in D}
Let $D$ be a Dedekind algebra. Then $D$ is the Dedekind completion of its Specker subalgebra $\spec(D)$.
\end{corollary}

\begin{proof}
By Proposition~\ref{prop: S dense in A}, $S$ is uniformly dense in $D$. Therefore, if $0 <d \in D$, then there is a sequence $\{s_n\}$ in $S$ converging to $d$ such that $0 \le s_n \le d$ (see, e.g., \cite[Lem.~3.16(1)]{BMO20c}). Thus, $S$ is essential in $D$, and hence $D$ is the Dedekind completion of $S$ by \cite[Prop.~2.12]{BMO16}. 
\end{proof}

The following corollary is a converse to Lemma~\ref{lem: Dedekind is Baer}. 

\begin{corollary} \label{cor: Baer ubal is Dedekind}
If $A \in \ubal$ is Baer, then $A$ is a Dedekind algebra.
\end{corollary}

\begin{proof}
Let $S$ be the Specker subalgebra of $A$ and let $D$ be the Dedekind completion of $A$. As we pointed out in the proof of Corollary~\ref{cor: S(D) essential in D}, $S$ is essential in $A$. Since $D$ is the Dedekind completion of $A$, we have that $A$ is essential in $D$ by \cite[Prop.~2.12]{BMO16}. Thus, $S$ is also essential in $D$, and so $D$ is the Dedekind completion of $S$. Because $A$ is Baer, 
$\Id(A)$ is complete, as pointed out in Remark~\ref{rem: complete idempotents}. Then $S$ is Baer by \cite[Thm.~4.3(2)]{BMMO15a}.
Therefore, 
$S$ is uniformly dense in $D$ by Corollary~\ref{cor: S dense in D}. Since $S\subseteq A\subseteq D$ and 
$S$ is uniformly dense in $D$, we also have that $A$ is uniformly dense in $D$. Because $A \in \ubal$, we conclude that $A = D$. Thus, $A$ is a Dedekind algebra. 
\end{proof}

Putting Lemma~\ref{lem: Dedekind is Baer} and Corollary~\ref{cor: Baer ubal is Dedekind} together, we obtain a direct choice-free proof of the following result in \cite{BMO13b}:

\begin{theorem}
Let $A\in\bal$. Then $A$ is a Dedekind algebra iff $A\in\ubal$ and $A$ is a Baer ring. 
\end{theorem}

\begin{remark} \label{rem: no nilpotents}
As promised earlier, we give a choice-free proof that each $A \in \bal$ has no nonzero nilpotent elements. If $a \in A$ with $a^n = 0$, then $|a|^n = 0$, so we may assume that $0 \le a$ with $a^n = 0$. Let $D$ be the Dedekind completion of $A$ and $S=\spec(D)$. Then $S$ is join-dense in $D$ by Corollary~\ref{cor: S(D) essential in D}. Therefore, if $a \ne 0$, there is $0 < b \in S$ with $b \le a$. Thus, $0 \le b^n \le a^n = 0$, so $b^n = 0$. Write $b = \sum_{i=1}^m r_i e_i$ in orthogonal form. Then $0 = b^n = \sum_{i=1}^m r_i^n e_i$. Multiplying by $e_i$ gives $r_i^n e_i = 0$, so $r_i = 0$ or $e_i = 0$ for each $i$. This implies $b = 0$, a contradiction. Consequently, $a = 0$ and hence $A$ has no nonzero nilpotents.
\end{remark}

Our next goal is to show that if $(S,\ll)\in\PBSp$, then $\ll$ extends to a KT-proximity on the Dedekind completion of $S$. For this we will utilize the Dieudonn\'{e} Lemma again. 

\begin{proposition} \label{prop: lifting proximity to Dedekind completion}
Let $(S, \ll)$ be a proximity Baer-Specker algebra and $D$ the Dedekind completion of $S$. If $\lhd$ is the closure of $\ll$ in $D \times D$, then $\lhd$ is a KT-proximity on $D$ and hence $(D,\lhd)$ is a KT-algebra.
\end{proposition}

\begin{proof}
By Corollary~\ref{cor: S dense in D}, $S$ is uniformly dense in $D$. Therefore, by Theorem~\ref{thm: Dieudonne}, $\lhd$ is a KT-proximity on $D$. Since $\lhd$ is closed by definition, $(D, \lhd) \in \cpda$.
\end{proof}

We next show how to lift weak proximity morphisms. For this we need the following well-known facts. The proof of (1) is straightforward, the proof of (2) is given in \cite[Prop.~II.3.7.13]{Bou89}, and the proof of (3) is given in  \cite[Thm.~II.3.6.2]{Bou89}.
\begin{enumerate}
\item If $\varphi : V_1 \to V_2$ is a function between normed vector spaces such that $\|\varphi(x) - \varphi(y)\| \le \|x-y\|$ for each $x,y \in V_1$, then $\varphi$ is uniformly continuous.
\item If $X$ is a complete metric space and $Y$ a dense subspace of $X$, then $X$ is (isometric to) the completion of $Y$.
\item Let $X, X'$ be complete metric spaces, $Y$ a dense subspace of $X$, and $Y'$ a dense subspace of $X'$. If $\varphi : Y \to Y'$ is a uniformly continuous map, then there is a unique extension of $\varphi$ to a uniformly continuous map $X \to X'$.
\end{enumerate}

\begin{lemma} \label{lem: morphisms are continuous}
Let $A, A' \in \bal$. If $\alpha : A \to A' $ is order preserving and $\alpha(a + r) = \alpha(a) + r$ for each $a \in A$ and $r \in \mathbb{R}$, then $\alpha$ is uniformly continuous. In particular, a weak proximity morphism is uniformly continuous.
\end{lemma}

\begin{proof}
Let $a,b \in A$ and set $\|a-b\| = \varepsilon$. Then $b - \varepsilon \le a \le b + \varepsilon$. By the hypotheses on $\alpha$ we have
\[
\alpha(b) - \varepsilon = \alpha(b - \varepsilon) \le \alpha(a) \le \alpha(b + \varepsilon) = \alpha(b) + \varepsilon.
\]
Therefore, $\|\alpha(a) - \alpha(b) \| \le \varepsilon = \|a - b\|$. Consequently, $\alpha$ is uniformly continuous.
\end{proof}

In the proof of the following proposition we will use several results from the Appendix.

\begin{proposition} \label{prop: lifting proximity morphism to Dedekind completion}
Let $\alpha : (S, \ll) \to (S', \ll')$ be a weak proximity morphism between proximity Baer-Specker algebras, let  $\lhd$ be the closure of $\ll$ in $D(S) \times D(S)$, and let $\lhd'$ be the closure of $\ll'$ in  $D(S') \times D(S')$. Then the  unique uniformly continuous extension $\beta:(D(S),\lhd) \rightarrow (D(S'),\lhd')$ of $\alpha$ is a proximity morphism.

\end{proposition}

\begin{proof}
We note that $\beta$ is well defined since $\alpha$ is uniformly continuous by Lemma~\ref{lem: morphisms are continuous} and $S$ is uniformly dense in $D(S)$ by Corollary~\ref{cor: S dense in D}. We then have $\beta(d) = \lim \alpha(a_n)$ for any sequence $\{a_n\}$ in $S$ converging to $d$. By Proposition~\ref{prop: lifting proximity to Dedekind completion}, the closure $\lhd$ of $\ll$ is a KT-proximity, and so is the closure $\lhd'$ of $\ll'$. We show that $\beta$ is a proximity morphism. 

(PM1) Since $\beta$ extends $\alpha$, we have $\beta(0) = \alpha(0) = 0$ and $\beta(1) = \alpha(1) = 1$.

(PM2) Let $c, d \in D(S)$. Say $c = \lim a_n$ and $d = \lim b_n$, where $\{a_n\}, \{b_n\} \subseteq S$. Then $c \wedge d = \lim (a_n \wedge b_n)$. Therefore, since $\alpha$ satisfies (PM2),
\begin{align*}
\beta(c \wedge d) &= \lim \alpha(a_n \wedge b_n) = \lim (\alpha(a_n) \wedge \alpha(b_n)) \\
&= \lim(\alpha(a_n)) \wedge \lim(\alpha(b_n)) = \beta(c) \wedge \beta(d).
\end{align*}

(PM3) Suppose that $c \lhd d$. Then there are sequences $\{a_n\}, \{b_n\}$ in $S$ with $c = \lim a_n$, $d = \lim b_n$, and $a_n \ll b_n$ for each $n$. We have $-\alpha(-a_n) \ll' \alpha(b_n)$. Taking limits yields $-\beta(-c) \lhd' \beta(d)$.

(PM4) We first show that $\beta(d) = \bigvee \{ \alpha(a) : a \in S, a \le d\}$. The inequality $\ge$ holds since $\beta$ is order preserving by (PM2) and extends $\alpha$. For the reverse inequality, we may write $d = \lim a_n$ with $a_n \le d$ for each $n$ (see, e.g., \cite[Lem.~3.16(1)]{BMO20c}). Then $\alpha(a_n)$ is below the join for each $n$, and so the limit is below the join. This yields the equality.
Therefore, by (PM4) applied to $\alpha$ in the second and third equalities below,
\begin{align*}
\beta(d) &= \bigvee \{ \alpha(a) : a \in S, a \le d\} = \bigvee \left\{ \bigvee \{ \alpha(b) : b \in S, b \ll a\} : a \in S, a \le d\right\} \\
&= \bigvee \{ \alpha(b) : b \in S, b \lhd d\}.
\end{align*}
From this it follows that $\beta(d) = \bigvee \{ \beta(c) : c \lhd d\}$.

(PM5) Let $0 < r \in \mathbb{R}$ and write $d = \lim a_n$. Since $\alpha$ satisfies (PM5), we have 
\[
\beta(rd) = \lim\alpha(ra_n) = \lim r\alpha(a_n) = r\beta(d).
\]

(PM6) Let $c, d \in D(S)$ with $c \lhd c$. There are sequences $\{a_n\}, \{b_n\}, \{b'_n\}$ in $S$ with $b_n \ll b'_n$ for each $n$ such that $d = \lim a_n$ and $c = \lim b_n = \lim b'_n$. By Lemma~\ref{lem: join and proximity morphism}(1), $\alpha(a_n \vee b_n) \le \alpha(a_n) \vee \alpha(b'_n)$. Therefore,
\begin{align*}
\beta(d \vee c) &= \lim \alpha(a_n \vee b_n) \le \lim (\alpha(a_n) \vee \alpha(b'_n)) \\
&= \lim \alpha(a_n) \vee \lim \alpha(b'_n) = \beta(d) \vee \beta(c) \le \beta(d \vee c).
\end{align*}
where the final inequality holds since $\beta$ is order preserving. Thus, $\beta(d \vee c) = \beta(d) \vee \beta(c)$.

(PM7) Let $c,d \in D(S)$ with $c \lhd c$. There are sequences $\{a_n\}, \{b_n\}, \{b'_n\}$ in $S$ with $b_n \ll b'_n$ for each $n$ such that $d = \lim a_n$ and $c = \lim b_n = \lim b'_n$. By Lemma~\ref{lem: join and proximity morphism}(2), $\alpha(a_n + b_n) \le \alpha(a_n) + \alpha(b'_n)$. Therefore,
\begin{align*}
\beta(d + c) &= \lim \alpha(a_n + b_n) \le \lim (\alpha(a_n) + \alpha(b'_n)) \\
&= \lim \alpha(a_n) + \lim \alpha(b'_n) = \beta(d) + \beta(c) \le \beta(d + c).
\end{align*}
where the final inequality holds by Lemma~\ref{lem: inequality}(1). Thus, $\beta(d + c) = \beta(d) + \beta(c)$.

(PM8) Let $ c,d \in D(S)$  with  $0 \leq c \lhd c$. We show $\beta(cd) = \beta(c)\beta(d)$. 
 By \cite[Rem.~8.9]{BMO16}, it suffices to prove this for $d \geq 0$.
There are sequences $\{a_n\}, \{b_n\}, \{b'_n\}$ of nonnegative elements in $S$ with $b_n \ll b'_n$ for each $n$ such that $d = \lim a_n$ and $c = \lim b_n = \lim b'_n$. By Lemma~\ref{lem: join and proximity morphism}(3), $\alpha(a_n  b_n) \le \alpha(a_n)  \alpha(b'_n)$. Therefore,
\begin{align*}
\beta(dc) &= \lim \alpha(a_n b_n) \le \lim (\alpha(a_n)  \alpha(b'_n)) \\
&= \lim \alpha(a_n)  \lim \alpha(b'_n) = \beta(d)  \beta(c) \le \beta(dc),
\end{align*}
where the final inequality holds by Lemma~\ref{lem: inequality}(2). Thus, $\beta(dc) = \beta(d) \beta(c)$. 
\end{proof}

\begin{proposition} \label{thm: functor D}
There is a functor $D : \PBSp \to \cpda$ that sends $(S, \ll)$ to $(D(S), \lhd)$ where $\lhd$ is the closure of $\ll$ in $D(S) \times D(S)$, and a proximity morphism $\alpha : (S, \ll) \to (S', \ll')$ to the unique continuous extension $D(\alpha) = \beta : (D(S),\lhd) \rightarrow (D(S'),\lhd')$ of $\alpha$.
\end{proposition}

\begin{proof}
By Propositions~\ref{prop: lifting proximity to Dedekind completion} and ~\ref{prop: lifting proximity morphism to Dedekind completion}, $D$ is well defined on objects and on morphisms. It is clear that $D$ sends identity maps to identity maps. To show that it preserves composition, let $\alpha_1 : (S_1, \ll_1) \to (S_2, \ll_2)$ and $\alpha_2 : (S_2, \ll_2) \to (S_3, \ll_3)$ be proximity morphisms between objects of $\PBSp$. Let 
$\beta_i$ be the unique continuous extension of $\alpha_i$ for $i = 1,2$. We need to show that 
$\beta_2 \star \beta_1$ is the unique continuous extension of $\alpha_2 \star \alpha_1$. For this it suffices to show that 
$\beta_2 \star \beta_1$ and $D(\alpha_2 \star \alpha_1)$ agree on $\Id(D(S_1))$, which is equal to $\Id(S_1)$ by Lemma~\ref{lem: S is Sp(D)}. For, if they agree on $\Id(S_1)$, then \cite[Lem.~6.4(2)]{BCMO22a} shows that they agree on $S_1$. Finally, as $S_1$ is uniformly dense in $D(S_1) $ by Corollary~\ref{cor: S dense in D} and both $\beta_2 \star \beta_1$ and $D(\alpha_2 \star \alpha_1)$ are continuous by Lemma~\ref{lem: morphisms are continuous}, they must agree on $D(S_1)$.

Let $f \in \Id(D(S_1))$. Then $(\beta_2\star \beta_1)(f) = \bigvee \{ \beta_2\beta_1(d) : d \in D(S_1), d \lhd f\}$. Since $\lhd$ is a KT-proximity, $(\beta_2\star \beta_1)(f) = \bigvee \{ \beta_2\beta_1(a) : a \in \dv{R}(D(S_1)), a \le f\}$. Fix $0 < \varepsilon \in \mathbb{R}$. If $a \le f$, then there is $e \in \Id(D(S_1))$ with $e \ll f$ and $a \le e + \varepsilon$ by Lemma~\ref{lem: M4}(2). Therefore, $\beta_2\beta_1(a)  \le \beta_2\beta_1(e) + \varepsilon$. Since this is true for each $a$ and $\varepsilon$, it follows that
\[
\bigvee \{ \beta_2\beta_1(a) : a \in \dv{R}(D(S_1)), a \le f\} \le \bigvee \{ \beta_2\beta_1(e) : e \in \Id(D(S_1)), e \ll f\}. 
\]
On the other hand, if $e \ll f$, there is $a \in \dv{R}(D(S_1))$ with $e \le a \le f$. Therefore, 
\[
\bigvee \{ \beta_2\beta_1(e) : e \in \Id(D(S_1)), e \ll f\} \le \bigvee \{ \beta_2\beta_1(a) : a \in \dv{R}(D(S_1)), a \le f\}. 
\]
Thus,
\begin{align*}
(\beta_2\star \beta_1)(f) &= \bigvee \{ \beta_2\beta_1(a) : a \in \dv{R}(D(S_1)), a \le f\} = \bigvee \{ \beta_2\beta_1(e) : e \in \Id(D(S_1)), e \ll f\} \\
&= \bigvee \{ \alpha_2\alpha_1(e) : e \in \Id(D(S_1)), e \ll f\}
= (\alpha_2 \star \alpha_1)(f) = D(\alpha_2 \star \alpha_1)(f). 
\end{align*}
Consequently, $\beta_2 \star \beta_1$ and $D(\alpha_2 \star \alpha_1)$ agree on $\Id(D(S_1))$, and the result follows. 
\end{proof}

We now prove one of the main results of the article. For this we recall that each element $s$ of a Specker algebra $S$ has a \emph{decreasing decomposition} $s = r_0 + \sum_{i=1}^n r_i e_i$ where $r_i \in \mathbb{R}$, $1 \ge e_1 \ge \cdots \ge e_n$ are idempotents of $S$, and $r_i \ge 0$ for $i \ge 1$ (see the Appendix). 

\begin{theorem} \label{thm: Id and D are equivalences} 
The functors $\Id : \cpda \to \dev$ and $D : \PBSp \to \cpda$ are equivalences, and the following diagram commutes up to natural isomorphism.
\[
\begin{tikzcd}
& \cpda \arrow[dl, "\Id"'] \\
\dev \arrow[dr, shift right = .2pc, "\Sp"'] & \\
& \PBSp \arrow[ul, shift right = .2pc, "\Id"'] \arrow[uu, "D"']
\end{tikzcd}
\]
\end{theorem}

\begin{proof}
To see that $\Id$ and $D$ are equivalences, by \cite[Thm.~IV.4.1]{Mac71} it is enough to show that $\Id$ and $D$ are full, faithful, and essentially surjective. We first consider $\Id$. Let $(B, \prec) \in \dev$. Set $D = D(\mathbb{R}[B]^\flat)$ and let $\lhd$ be the closure of $\prec^\flat$ in $D$. Since $(\mathbb{R}[B]^\flat, \prec^\flat) \in \PBSp$ by \cite[Thm.~5.11(1)]{BCMO22a},  $(D, \lhd) \in \cpda$ by Proposition~\ref{prop: lifting proximity to Dedekind completion}. By Theorem~\ref{thm: prox}(1), $\lhd$ restricts to a proximity $\prec$ on $\Id(D)$ such that if $e, f \in B$, then $ e^\flat \lhd f^\flat$ iff $e \prec f$ (see the equivalence~(\ref{eqn})).  Therefore, $(B, \prec)$ is isomorphic to $\Id(D, \lhd)$. Thus, $\Id$ is essentially surjective.

To show that $\Id$ is full, let $\sigma : \Id(D, \lhd) \to \Id(D', \lhd')$ be a de Vries morphism. The proximity $\lhd$ restricts to a proximity $\prec$ on $\Id(D)$ by Theorem~\ref{thm: prox}(1), and the same is true for $\lhd'$ and $\Id(D')$. Also, $\prec$ extends to a proximity $\ll$ on $\spec(D)$ by \cite[Cor.~5.8]{BCMO22a}, and the same is true for $\prec'$ and $\spec(D')$. Then $\sigma$ extends (uniquely) to a proximity morphism $\alpha : (\spec(D), \ll) \to (\spec(D'), \ll')$ by \cite[Cor.~6.6]{BCMO22a}.  Proposition~\ref{prop: lifting proximity morphism to Dedekind completion} shows that $\alpha$ extends to a proximity morphism $\beta : (D, \lhd) \to (D', \lhd')$ since $\spec(D)$ is uniformly dense in $D$ by Corollary~\ref{cor: S dense in D}, and the same is true for $\spec(D')$. Therefore, $\Id(\beta) = \beta|_{\Id(D)} = \sigma$. Thus, $\Id$ is a full functor.

To see that $\Id$ is faithful, let $\beta, \beta' : (D, \lhd) \to (D', \lhd')$ be proximity morphisms which agree on $\Id(D)$. Using decreasing decompositions, it follows from Lemma~\ref{lem: technical properties}(6) that $\beta, \beta'$ agree on the Specker subalgebra $S$ of $D$. Since $\beta, \beta'$ are continuous and $S$ is uniformly dense in $D$ by Corollary~\ref{cor: S dense in D}, we see that $\beta = \beta'$. Therefore, $\Id$ is faithful, hence $\Id$ is an equivalence.

Next, we consider $D$. To see it is essentially surjective, let $(D, \lhd) \in \cpda$. Set $B = \Id(D)$ and $S$ to be the Specker subalgebra of $D$. Then $\Id(S) = B$. We have that $D$ is Baer by Lemma~\ref{lem: Dedekind is Baer}, and so $B$ is complete by Remark~\ref{rem: complete idempotents}. Moreover, $\lhd$ restricts to a de Vries proximity $\prec$ on $B$ by Theorem~\ref{thm: prox}(1). In addition, $\prec$ lifts to a proximity $\ll$ on $S$ by \cite[Cor.~5.8]{BCMO22a}. We claim that $\lhd$ is the closure of $\ll$. Since $D = D(S)$ by Corollary~\ref{cor: S(D) essential in D}, this will yield that $(D, \lhd) = D(S, \ll)$. To see this, let $s, t \in S$. Write $s = r_0 + \sum_{i=1}^n r_i e_i$ and $t = r_0 + \sum_{i=1}^n r_i f_i$ in compatible decreasing form, and set $p_i = r_0 + \cdots + r_i$ for $1 \le i \le n$ as in Lemma~\ref{lem: technical properties}(3).

\begin{claim} \label{claim: lhd and ll}
$s \lhd t$ iff $s \ll t$ iff $e_i \prec f_i$ for each $i$. 
\end{claim}

\begin{proofclaim}
Let $s \lhd t$. Then $[(s-p_{i-1}) \wedge r_i] \vee 0 \lhd [(t-p_{i-1}) \wedge r_i] \vee 0$ for each $i$. Therefore, $r_i e_i \lhd r_i f_i$ by Lemma~\ref{lem: technical properties}(2). Since $r_i > 0$,  we conclude that $e_i \lhd f_i$ for $i \ge 1$. Because $\prec$ is the restriction of $\lhd$ to $\Id(D)$, we have $e_i \prec f_i$. A similar argument yields that $s \ll t$ implies $e_i \prec f_i$ for each $i$. The converse implications are easy to see by applying (P1), (P6), and (P7).
\end{proofclaim}

Thus, $\lhd$ restricts to $\ll$ on $S$. Since $\lhd$ is a closed proximity, the closure $\lhd'$ of $\ll$ is contained in $\lhd$. Let $d,e \in D$ with $d \lhd e$. Since $S$ is uniformly dense in $D$, we may write $d = \lim a_n$ and $e = \lim b_n$ for some sequences $\{a_n\}, \{b_n\}$ in $S$. By \cite[Lem~3.16(1)]{BMO20c}, we may assume that $a_n \le d$ and $e \le b_n$ for each $n$.
This yields $a_n \le d \lhd e \le b_n$, so $a_n \lhd b_n$. Therefore, $a_n \ll b_n$ for each $n$, and so $d \lhd' e$. Consequently, $\lhd$ is equal to the closure of $\ll$. Thus, $D$ is essentially surjective.

To see that $D$ is full, let $\beta : (D(S), \ll) \to (D(S'), \ll')$ be a proximity morphism. We show that $\alpha := \beta|_{S}$ is a function from $S$ to $S'$. First, $\beta|_{\Id(D(S))} : \Id(D(S)) \to \Id(D(S'))$ is a de Vries morphism by Theorem~\ref{thm: prox}(2). Next, let $s \in S$, and write $s = a_0 + \sum_i b_i e_i$ in decreasing form. Then $\beta(s) = a_0 + \sum_i b_i \beta(e_i)$ by Lemma~\ref{lem: technical properties}, so $\beta(s) \in S'$ since each $\beta(e_i) \in \Id(D(S')) = \Id(S')$. Therefore, $\alpha$ is a well-defined function. To show that $\alpha$ is a proximity morphism, by Proposition~\ref{prop:weak proximity morphism}, it suffices to show that $\alpha$ is a weak proximity morphism. All the axioms except (PM4) are straightforward to see. To verify (PM4), let $b \in S$. Then $\alpha(b) = \bigvee \{ \beta(c) : c \in D(S), c \lhd b\}$. Let $c \in D(S)$ with $c \lhd b$. There is a sequence $\{a_n\}$ in $S$ with $a_n \le c$ such that $\lim a_n = c$. Since $\beta$ is continuous, $\lim \alpha(a_n) = \beta(c)$, and therefore $\bigvee \alpha(a_n) = \beta(c)$ by \cite[Lem.~3.16(2)]{BMO20c}. Consequently, $\alpha(b) = \bigvee \{ \alpha(a) : a \in S, a \ll b\}$, which verifies (PM4).

To see that $D$ is faithful, let $\alpha, \alpha' : (S, \ll) \to (S', \ll')$ be proximity morphisms with $D(\alpha) = D(\alpha')$. Since $D(\alpha)$ extends $\alpha$ and $D(\alpha')$ extends $\alpha'$, we have that $\alpha = \alpha'$. Therefore, $D$ is faithful. Thus, $D$ is an equivalence.

Finally, to see that the diagram commutes up to natural isomorphism, we show that $\Id \circ D \circ \Sp$ is naturally equivalent to the identity functor on $\dev$. Let $(B, \prec) \in \dev$. Then 
\[
D(\Sp(B, \prec)) = D(\mathbb{R}[B]^\flat, \prec^\flat) = (D(\mathbb{R}[B]^\flat), \lhd),
\] 
where $\lhd$ is the closure of $\prec^\flat$. The functor $\Id$ then sends this to $(\Id(D(\mathbb{R}[B]^\flat)), \prec')$, where $\prec'$ is the restriction of $\lhd$ to $\Id(D(\mathbb{R}[B]^\flat))$. As seen above, the boolean isomorphism $\tau_B :  B \to \Id(D(\mathbb{R}[B]^\flat))$ (see \cite[Rem.~4.10]{BCMO22a}) satisfies $e \prec f$ iff $e^\flat \prec^\flat f^\flat$, iff $e \lhd f$. The proof of \cite[Thm.~6.9]{BCMO22a} shows that $\tau$ is then a natural isomorphism between the identity functor and $\Id \circ D \circ \Sp$.
\end{proof}

As mentioned in Remark~\ref{rem: S is a functor}, we finish the section by showing that $\spec : \pda \to \PBSp$ is a functor naturally isomorphic to $\Sp \circ \Id$.

\begin{remark} \label{rem: S is a functor}
We define $\spec$ by sending $(D, \lhd) \in \pda$ to $(\spec(D), {\lhd}|_{\spec(D)})$ and a proximity morphism $\alpha : (D, \lhd) \to (D', \lhd')$ to $\alpha|_{\spec(D)}$. Set ${\prec} = {\lhd}|_{\Id(D)}$ and $\ll$ to be the lift of $\prec$ to $\spec(D)$. By Claim~\ref{claim: lhd and ll}, $\ll$ is the restriction of $\lhd$ to $\spec(D)$, and hence $(\spec(D), {\lhd}_{\spec(D)}) \in \PBSp$. 

Let $\alpha : (D, \lhd) \to (D', \lhd')$ be a proximity morphism. By Theorem~\ref{thm: prox}(2), $\alpha$ sends idempotents to idempotents. If $s \in \spec(D)$, then we can write $s = r_0 + \sum_i r_i e_i$ in decreasing form as discussed in the appendix.
 It then follows from the proof of \cite[Lem.~6.4(2)]{BCMO22a} that $\alpha(s) = r_0 + \sum_i r_i\alpha(e_i)$. Thus, $\alpha|_{\spec(D)}$ is a well-defined function. The axioms (PM1)-(PM3) are straightforward, and the argument to show that $D$ is full in the proof of Theorem~\ref{thm: Id and D are equivalences} yields that $\alpha|_{\spec(D)}$ satisfies (PM4). The same argument can be used to show that $\spec$ preserves composition. Consequently, $\spec$ is a covariant functor. 

By \cite[Prop.~4.11]{BCMO22a}, there is an $\ell$-algebra isomorphism $(-)^\flat : \spec(D) \to \mathbb{R}[\Id(D)]^\flat$. Furthermore, by \cite[Thm.~5.11]{BCMO22a}, if $s, t \in \spec(D)$, then $s \ll t$ iff $s^\flat \prec^\flat t^\flat$. Therefore, from \cite[Lem.~8.3]{BMMO15b} we have that $(-)^\flat : (\spec(D), \ll) \to (\Sp\Id(D), \prec^\flat)$ is a proximity isomorphism.
If we define $\rho : \spec \to \Sp \circ \Id$ by setting $\rho_{(D, \lhd)}$ to be the $\ell$-algebra isomorphism $(-)^\flat : \spec(D) \to \mathbb{R}[\Id(D)]^\flat$ for each $(D, \lhd) \in \pda$, then a straightforward argument shows that $\rho$ is a natural transformation, and hence it is a natural isomorphism since each $\rho_{(D, \lhd)}$ is an isomorphism.
\end{remark}

\section{Putting everything together} \label{sec: summary}

In this final section we summarize our main results. We have given direct choice-free proofs of the following equivalences:

\begin{enumerate}
\item $\ubal$ is equivalent to $\cpda$ (Theorem~\ref{thm: ubal = KT}).
\item $\cpda$ is equivalent to $\dev$ (Theorem~\ref{thm: Id and D are equivalences}).
\item $\cpda$ is equivalent to $\PBSp$ (Theorem~\ref{thm: Id and D are equivalences}).
\end{enumerate}

Thus, we arrive at the following diagram.

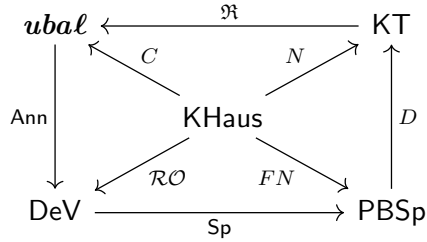
\[
\begin{tikzcd}
\ubal \arrow[dd, "\ann"'] && \cpda  \arrow[ll, "\dv{R}"'] \\
& \KHaus \arrow[ul, "C"'] \arrow[ur, "N"] \arrow[dl, "\RO"] \arrow[dr, "FN"'] & \\
\dev \arrow[rr, "\Sp"'] && \PBSp \arrow[uu, "D"']
\end{tikzcd}
\]
We conclude by showing that the diagram commutes (up to natural isomorphism).  

To see that the outside diagram commutes, let $(D, \lhd) \in \cpda$ and $A = \dv{R}(D, \lhd)$. Then ${(\ann(A), \prec)} \cong (\Id(D), \prec)$ by Theorem~\ref{thm: prox}(1). Therefore, $\Sp(\ann(A))$ is isomorphic to the Specker subalgebra  $\spec(D)$ of $D$, and hence $D(\Sp(\ann(A)) \cong D$ by Corollary~\ref{cor: S(D) essential in D}. Moreover, the unique lift of the proximity $\prec$ on $\Id(D)$ to a proximity $\ll$ on $\spec(D)$ is equal to $\lhd|_{\spec(D)}$ by Remark~\ref{rem: S is a functor}. The closure of $\ll$ is a KT-proximity on $D$ by Proposition~\ref{prop: lifting proximity to Dedekind completion}. Since the closure of $\ll$ and $\lhd$ restrict to $\prec$ on $\Id(D)$, they are equal by Theorem~\ref{thm: Id and D are equivalences}. Thus, the outside diagram commutes.

To see that the inside of the diagram commutes, let $X \in \KHaus$. Then the corresponding de Vries algebra is $(\RO(X), \prec)$, where $\prec$ is given by $U \prec V$ iff $\cl(U) \subseteq V$ for each $U, V \in \RO(X)$. The corresponding Baer-Specker algebra is ${(FN(X), \ll)} \in \PBSp$, where $\ll$ is the unique lift of $\prec$ when we identify $\RO(X)$ with $\Id(FN(X))$. The corresponding $\ubal$-algebra is $C(X)$, and the corresponding proximity Dedekind algebra is $(N(X), \lhd) \in \cpda$, where $\lhd$ on $N(X)$ is given by $f \lhd g$ iff there is $c \in C(X)$ with $f \le c \le g$. Furthermore, $\ll$ is the restriction of $\lhd$ to $FN(X)$. Remark~\ref{rem: ann = RO} shows that $\ann\circ C \cong \RO$. The Kat\v{e}tov-Tong theorem shows that $\dv{R} \circ N = C$. By Proposition~\ref{thm: functor D}, $D \circ FN \cong N$ since $FN(X)$ is the Specker subalgebra of $N(X)$, so $N(X) \cong D(FN(X))$ by Corollary~\ref{cor: S(D) essential in D}. That $\Sp \circ \RO \cong FN$ follows from Remark~\ref{rem: FN = Sp Ro}.

\appendix

\section*{Appendix: Weak proximity morphisms}
\renewcommand{\thetheorem}{A.\arabic{theorem}}
\setcounter{theorem}{0}

As promised in Remark~\ref{rem: wpm=pm}(1), we  prove that a weak proximity morphism between proximity Baer-Specker algebras is always a proximity morphism. This is utilized in Theorem~\ref{thm: Id and D are equivalences}, which is one of our main results. Our proof that each weak proximity morphism is a proximity morphism requires a series of technical lemmas.

\begin{lemma} \label{lem: inequality}
Let $\alpha : (S, \ll) \to (S', \ll')$ be a weak proximity morphism between proximity Baer-Specker algebras and $a, b \in S$.
\begin{enumerate}[$(1)$]
\item $\alpha(a) + \alpha(b) \le \alpha(a + b)$.
\item If $0 \le a, b$, then $\alpha(a) \alpha(b) \le \alpha(a b)$.
\end{enumerate}
\end{lemma}

\begin{proof}
The proofs of (1) and (2) are similar, and we only prove (1). By \cite[Prop.~5.1]{BMMO15b}, the restrictions of $\ll$ and $\ll'$ to idempotents are de Vries proximities, $\sigma = \alpha|_{\Id(S)} : (\Id(S), \prec) \to (\Id(S'), \prec')$ is a de Vries morphism, and we have the following commutative diagram by \cite[Cor.~6.6]{BCMO22a}, where the vertical maps are $\bal$-isomorphisms. 
\[
\begin{tikzcd}[column sep = 5pc]
S \arrow[r, "\alpha"] \arrow[d, "(-)^\flat"'] & S' \arrow[d, "(-)^\flat"] \\
\mathbb{R}[\Id(S)]^\flat \arrow[r, "\sigma^\flat"'] & \mathbb{R}[\Id(S')]^\flat
\end{tikzcd}
\]
It then suffices to show that the inequality in (1) holds for $\sigma^\flat$. For this, let $a, b \in \mathbb{R}[\Id(S)]^\flat$.  Recalling the operations on $\mathbb{R}[\Id(S)]^\flat$ given after Definition~\ref{star definition}, if $r \in \mathbb{R}$, then
\begin{align*}
(\sigma^\flat(a) + \sigma^\flat(b))(r) &= \bigvee \{ \sigma^\flat(a)(s) \wedge \sigma^\flat(b)(t) : s + t \ge r \} \\
&= \bigvee \{ \sigma(a(s)) \wedge \sigma(b(t)) : s + t \ge r \}
\end{align*}
and
\begin{align*}
(\sigma^\flat(a + b))(r) &= \sigma((a+b)(r)) = \sigma\left(\bigvee \{ a(s) \wedge b(t) : s + t \ge r\} \right) \\
&\ge \bigvee \{ \sigma(a(s) \wedge b(t)) : s + t \ge r \} \\
&= \bigvee \{ \sigma(a(s)) \wedge \sigma(b(t)) : s + t \ge r \}.
\end{align*}
Thus, since $ (\sigma^\flat(a) + \sigma^\flat(b))(r) \le \sigma^\flat(a+b)(r)$ for each $r \in \mathbb{R}$, we have that $ \sigma^\flat(a) + \sigma^\flat(b) \le \sigma^\flat(a + b)$.
\end{proof}

\begin{lemma} \label{lem: de Vries fact} 
Let $\sigma : (B, \prec) \to (B', \prec')$ be a de Vries morphism between de Vries algebras. If $e, f, g \in B$ with $f \prec g$, then $\sigma(e \vee f) \le \sigma(e) \vee \sigma(g)$.
\end{lemma}

\begin{proof}
Since $f \prec g$ we have $\sigma(f^*)^* \prec \sigma(g)$, so $\sigma(f^*)^* \le \sigma(g)$. This yields $\sigma(g)^* \le \sigma(f^*)$. Therefore,
\begin{align*}
\sigma(e \vee f) \wedge \sigma(g)^* &\le \sigma(e \vee f) \wedge \sigma(f^*) = \sigma((e \vee f) \wedge f^*)
= \sigma(e \wedge f^*) \le \sigma(e).
\end{align*}
From this it follows that $\sigma(e \vee f) \le \sigma(e) \vee \sigma(g)$.
\end{proof}

By \cite[Sec.~5]{BMMO15b}, from an orthogonal decomposition $a = \sum_{i=0}^n r_i e_i$ we can obtain a decreasing decomposition as follows. Without loss of generality we may assume that $r_0 \le \cdots \le r_n$. Then we can write
\[
a = r_0 (e_0 + \cdots + e_n) + (r_1-r_0)(e_1 + \cdots + e_n) + \cdots + (r_n - r_{n-1})e_n.
 \]
Therefore, $a = \sum_{i=0}^n p_if_i$, where $p_0 = r_0$,  $p_i=r_i-r_{i-1}$ for $i\ge 1$, and $f_i=\sum_{j=i}^n e_j=\bigvee_{j=i}^n e_j$ (the latter equality follows from \cite[Eqn.~XIII.3(14)]{Bir79}). This exhibits $a$ as a linear combination of a sequence of decreasing idempotents. Moreover, by eliminating coefficients that are 0, we may assume that all the coefficients are nonzero and all of them except possibly $p_0$ are positive. Furthermore, if $a=\sum_{i=0}^n r_i e_i$ is a full orthogonal decomposition of $a$, then $f_0 = 1$. In this case we will write the corresponding decreasing decomposition as $a = p_0 + \sum_{i=1}^n p_i f_i$.

In order to prove Lemmas~\ref{lem: reflexive invertible elements} and \ref{rem: some properties}, we require the following result.

\begin{lemma} \label{lem: technical properties}
Let $S$ be a Specker algebra.
\begin{enumerate}[$(1)$]
\item \cite[Lem.~4.9(5)]{BMMO15b} If $0 \ne e \in \Id(S)$ and $r \in \mathbb{R}$ with $re \ge 0$, then $r \ge 0$.

\item \cite[Lem.~4.9(6)]{BMMO15b} Let $0 \ne e, f \in \Id(S)$ and $0 < r, p \in \mathbb{R}$. Then $re \le pf$ iff $r \le p$ and $e \le f$.

\item \cite[Lem.~5.4(1)]{BCMO22a} Let $a \in S$. If $r, p \in \mathbb{R}$ with $r < p$, then $(a \wedge p) - (a \wedge r) = [(a - r) \wedge (p - r)] \vee 0$.

\item \cite[Lem.~5.4(1)]{BCMO22a} Let $a \in S$ with $a = r_0 + \sum_{i=1}^n r_i e_i$ in decreasing form. Set $p_i = r_0 + \cdots + r_i$ for $1 \le i \le n$. Then $[(a-p_{i-1}) \wedge r_i] \vee 0 =r_i e_i$.

\item \cite[Lem.~5.4(2)]{BCMO22a} Let $a, b\in S$. Then there exist $r_0 < \cdots < r_n$ in $\mathbb R$ with $r_0\le s,t\le r_n$ such that $a$ and $b$ have decreasing decompositions $a = r_0 + \sum_{i=1}^n (r_i-r_{i-1})e_i$ and $b = r_0 + \sum_{i=1}^n (r_i-r_{i-1})f_i$. Moreover, if $a, b \ge 0$, then we may assume $r_0 = 0$.

\item \cite[Lem.~6.4(2)]{BCMO22a} 
Suppose $\alpha : (S, \ll) \to (S', \ll')$ is a weak proximity morphism between proximity Baer-Specker algebras.  If $a  = r_0 + \sum_i r_i e_i$ is in decreasing form, then $\alpha(a) = r_0 + \sum_i r_i \alpha(e_i)$. 
\end{enumerate}
\end{lemma}

\begin{lemma} \label{lem: reflexive invertible elements}
Suppose $\alpha : (S, \ll) \to (S', \ll')$ is a weak proximity morphism between proximity Baer-Specker algebras. Let $0 \le c \in S$.
\begin{enumerate}[$(1)$]
\item $c$ is invertible iff there is $0 < r \in \mathbb{R}$ with $r \le c$.
\item If $0 \le b \ll c$ and $b$ is invertible, then $\alpha(c)$ is invertible and $\alpha(c)^{-1} \le \alpha(b^{-1})$.
\end{enumerate}
\end{lemma}

\begin{proof}
(1) Suppose $0 < r \le c$ for some $r \in \mathbb{R}$. 
Write $c = r_0 e_0 + \cdots + r_n e_n$ in full orthogonal form. Then $re_i \le ce_i = r_i e_i$, which implies that $r \le r_i$ by Lemma~\ref{lem: technical properties}(2). Consequently, each $r_i \ne 0$, and hence $r_0^{-1} e_0 + \cdots + r_n^{-1} e_n$ is the multiplicative inverse of $c$. 

Conversely, let $c$ be invertible, and write $c = r_0 e_0 + \cdots + r_n e_n$ as above. Without loss of generality suppose that $r_0 \le r_i$ for each $i$.  Since $0 \le c, e_i$ we have $0 \le ce_i = r_ie_i$, so $r_i \ge 0$ by Lemma~
\ref{lem: technical properties}(1). If $r_0 = 0$, then $ce_0 = r_0 e_0 = 0$, which is false since $c$ is invertible and $e_0 \ne 0$. Therefore,
\[
c \ge r_0e_0 + \cdots + r_0 e_n = r_0(e_0 + \cdots + e_n) = r_0(e_0 \vee \cdots \vee e_n) = r_0.
\]

(2) Write $b = r_0 + \sum_{i=1}^n r_i e_i$ and $c = r_0 + \sum_{i=1}^n r_i f_i$ in compatible decreasing form by Lemma~\ref{lem: technical properties}(5). We may assume that $1 > e_1, f_1$. Since $e_1 \ge e_i$, we have $e_1^*e_i = 0$ for each $i$. Therefore, if $r_0 = 0$, then $be_1^* = 0$, so $e_1^* = 0$ as $b$ is invertible. This forces $e_1 = 1$, which is false by assumption. So, $r_0 \ne 0$. In addition, since $b \ge 0$, we have $r_0e_1^* = be_1^* \ge 0$, which implies that $r_0 > 0$ by Lemma~\ref{lem: technical properties}(1). Set $p_0 = r_0$ and $p_i = r_0 + r_1 +  \cdots + r_i$ for $i \ge 1$. As each $r_i \ge 0$ for $i \ge 1$, all $p_i > 0$. 
Therefore,
\begin{align*}
b &= (r_0 + r_1 + \cdots + r_n)e_n + (r_0 + r_1 + \cdots + r_{n-1})(e_{n-1} - e_n) + \cdots + (r_0 + r_1)(e_1 - e_2) + r_0(1-e_1) \\
&= p_ne_n + p_{n-1}(e_{n-1} - e_n) + \cdots + p_1(e_1 - e_2) + p_0(1 - e_1)
\end{align*}
is in full orthogonal form. Consequently, since $b$ is invertible and all $p_i \ne 0$, 
\[
b^{-1} =  p_n^{-1}e_n + p_{n-1}^{-1}(e_{n-1} - e_n) + \cdots + p_1^{-1}(e_1 - e_2) + p_0^{-1}(1 - e_1).
\]
Because $0 < p_0 \le \cdots \le p_n$, we have $p_n^{-1} \le \cdots \le p_0^{-1}$. From this we may write $b^{-1}$ in decreasing form as
\begin{align*}
b^{-1} &= p_n^{-1}(e_n + e_{n-1} - e_n + \cdots + 1 - e_1) + (p_{n-1}^{-1} - p_n^{-1})(e_{n-1} - e_n + \cdots + 1 - e_1) + \cdots + p_0^{-1}(1-e_1) \\
&= p_n^{-1} + (p_{n-1}^{-1} - p_{n-2}^{-1})e_n^* + \cdots + p_0^{-1}e_1^*.
\end{align*}
Thus, by Lemma~\ref{lem: technical properties}(6),
\[
\alpha(b^{-1}) = p_n^{-1} + (p_{n-1}^{-1} - p_{n-2}^{-1})\alpha(e_n^*) + \cdots + p_0^{-1}\alpha(e_1^*).
\]
Since $b$ is invertible, there is $0 < r \in \mathbb{R}$ with $r \le b$ by (1). Because $b \ll c$, (P2) implies that $r \le c$. Therefore, $r \le \alpha(c)$, and so $\alpha(c)$ is invertible by (1). Since $\alpha(c) = r_0 + \sum_{i=1}^n r_i \alpha(f_i)$ by Lemma~\ref{lem: technical properties}(6), a similar calculation applied to $\alpha(c)$ yields
\[
\alpha(c)^{-1} = p_n^{-1} + (p_{n-1}^{-1} - p_{n-2}^{-1})\alpha(f_n)^* + \cdots + p_0^{-1}\alpha(f_1)^*.
\]
From $b \ll c$ we get $e_i \ll f_i$ for each $i$ by Claim~\ref{claim: lhd and ll}. Therefore, $-\alpha(-e_i) \ll \alpha(f_i)$, and so $\alpha(e_i^*)^* \ll \alpha(f_i)$ by Equation~(\ref{5.1}). Taking complements gives $\alpha(f_i)^* \ll \alpha(e_i^*)$, and so $\alpha(f_i)^* \le \alpha(e_i^*)$ for each $i$. Thus,  $\alpha(c)^{-1} \le \alpha(b^{-1})$.
\end{proof}

\begin{lemma} \label{rem: some properties}
Let $(S, \ll)$ be a proximity Baer-Specker algebra and $a, b  \in S$. If $a = r_0 + \sum_i r_i e_i$ and $b = r_0 + \sum_i r_i f_i$ are in decreasing form, then
\begin{enumerate}[$(1)$]
\item $a \vee b = r_0 + \sum_i r_i (e_i \vee f_i)$,
\item $a \wedge b = r_0 + \sum_i r_i (e_i \wedge f_i)$.
\end{enumerate}
\end{lemma}

\begin{proof}
We prove (1); the proof of (2) is similar.  For $1 \le i \le n$ set $p_i = r_0 + \cdots + r_i$. By Lemma~\ref{lem: technical properties}(3), we have $[(a- p_{i-1}) \wedge r_i]\vee 0 = r_i e_i$ and $[(b - p_{i-1}) \wedge r_i] \vee 0 = r_i f_i$. Therefore, by standard vector lattice identities,
\begin{align*}
[((a \vee b) - p_{i-1}) \wedge r_i] \vee 0 &=  [((a-p_{i-1}) \vee (b - p_{i-1})) \wedge r_i] \vee 0 \\
&= ([(a - p_{i-1}) \wedge r_i] \vee [(b - p_{i-1}) \wedge r_i]) \vee 0 \\
&= ([(a - p_{i-1}) \wedge r_i] \vee 0) \vee ([(b - p_{i-1}) \wedge r_i] \vee 0) \\
&= r_ie_i \vee r_if_i = r_i(e_i \vee f_i),
\end{align*}
where the last equality holds since $r_i \ge 0$. Because $r_0 \le a, b \le r_n$, we have $p_0 \le a \vee b \le p_n$. Therefore, $(a \vee b) \wedge p_0 = p_0$ and $(a \vee b) \wedge p_n = a \vee b$. By the calculation above and Lemma~\ref{lem: technical properties}(3),
\begin{align*}
(a \vee b) - p_0 &= (a \vee b) \wedge p_n - (a \vee b) \wedge p_0 = \sum_{i=1}^n \left[(a \vee b) \wedge p_i - (a \vee b) \wedge p_{i-1} \right] \\
&=  \sum_{i=1}^n \left( [((a \vee b) - p_{i-1}) \wedge r_i] \vee 0 \right) = \sum_{i=1}^n r_i(e_i \vee f_i).
\end{align*}
Since $r_0 = p_0$, it follows that $a \vee b = r_0 + \sum_i r_i (e_i \vee f_i)$.
\end{proof}

\begin{lemma} \label{lem: join and proximity morphism}
Let $\alpha : (S, \ll) \to (S', \ll')$ be a weak proximity morphism between proximity Baer-Specker algebras. Suppose that $a, b, c\in S$ with $b \ll c$.
\begin{enumerate}[$(1)$]
\item $\alpha(a \vee b) \le \alpha(a) \vee \alpha(c)$.
\item $\alpha(a + b) \le \alpha(a) + \alpha(c)$.
\item If $0 \le a, b$, then $\alpha(ab) \le \alpha(a)\alpha(c)$.
\end{enumerate}
\end{lemma}

\begin{proof}
(1) By Lemma~\ref{lem: technical properties}(5) we may write 
$a = r_0 + \sum_{i=1}^n r_i e_i$, $b = r_0 + \sum_{i=1}^n r_i f_i$, and $c = r_0 + \sum_{i=1}^n r_i g_i$ in decreasing form. Then $a \vee b = r_0 + \sum_{i=1}^n r_i (e_i\vee f_i)$ by Lemma~\ref{rem: some properties}(1). Therefore, 
$\alpha(a \vee b) = r_0 + \sum_{i=1}^n r_i \alpha(e_i \vee f_i)$ by Lemma~\ref{lem: technical properties}(6). Since $\alpha(e_i \vee f_i) \le \alpha(e_i) \vee \alpha(g_i)$ by Lemma~\ref{lem: de Vries fact}, we have
\begin{align*}
\alpha(a \vee b) &= r_0 + \sum_{i=1}^n r_i \alpha(e_i \vee f_i) \\
&\le r_0 + \sum_{i=1}^n (r_i \alpha(e_i) \vee r_i\alpha(g_i)) = \alpha(a) \vee \alpha(c),
\end{align*}
where the last equality follows from Lemma~\ref{rem: some properties}(1).

(2) Since $b \ll c$, we have $-\alpha(-b) \le \alpha(c)$ by (PM3) and (P2), so $-\alpha(c) \le \alpha(-b)$. Therefore, by Lemma~\ref{lem: inequality}(1), 
\[
\alpha(a+b) - \alpha(c) \le \alpha(a + b) + \alpha(-b) \le \alpha((a+b) + (-b)) = \alpha(a).
\] 
This yields $\alpha(a+b) \le \alpha(a) + \alpha(c)$.

(3) First suppose that $b$ is invertible.  Since $0 \le b$, Lemma~\ref{lem: reflexive invertible elements}(1) shows that $0 < r \le b$ for some $r \in \mathbb{R}$. Because $b \ll c$, we also have $r \le c$, and so $c$ is invertible.  
By Lemma~\ref{lem: reflexive invertible elements}(2), $\alpha(c)^{-1} \le \alpha(b^{-1})$. Consequently, 
\[
\alpha(ab) \alpha(c)^{-1} \le \alpha(ab)\alpha(b^{-1}) \le \alpha((ab)b^{-1}) = \alpha(a)
\]
by Lemma~\ref{lem: inequality}(2). Multiplying by $\alpha(c)$ yields $\alpha(ab) \le \alpha(a)\alpha(c)$.

For an arbitrary $b\ge 0$, by Lemma~\ref{lem: reflexive invertible elements}(1), $1 + b$ is invertible, and $1 + b \ll 1 + c$. Therefore, by the previous case, $\alpha(a(1 + b)) \le \alpha(a)\alpha(1 + c)$. Since $\alpha$ is a weak proximity morphism,
\[
\alpha(a + ab) = \alpha(a(1+b)) \le \alpha(a)\alpha(1+c) = \alpha(a)(1 + \alpha(c)) = \alpha(a) + \alpha(a)\alpha(c).
\] 
By Lemma~\ref{lem: inequality}(1), 
\[
\alpha(a) + \alpha(ab) \le \alpha(a + ab) \le \alpha(a) + \alpha(a)\alpha(c).
\]
Subtracting $\alpha(a)$ yields (3).
\end{proof}

\begin{remark}
Lemma~\ref{lem: join and proximity morphism}(2) follows from \cite[Lem.~7.1(2)]{BMMO15b}, the proof of which is not choice-free.
\end{remark}

We are ready to prove the main result of the Appendix.

\begin{theorem}\label{prop:weak proximity morphism}
Let $(S, \ll)$ and $(S', \ll')$ be proximity Baer-Specker algebras. A map $\alpha : S \to S'$ is a proximity morphism iff it is a weak proximity morphism.
\end{theorem}

\begin{proof}
Clearly if $\alpha$ is a proximity morphism, then it is a weak proximity morphism. For the converse, we only need to show that axioms (PM6)--(PM8) hold.
Let $a, c \in S$ with $c \ll c$. 

(PM6) The inequality $\alpha(a \vee c) \ge \alpha(a) \vee \alpha(c)$ holds since $\alpha$ is order preserving. The reverse inequality holds by Lemma~\ref{lem: join and proximity morphism}(1).

(PM7) The inequality $\alpha(a+c) \ge \alpha(a) + \alpha(c)$ holds by Lemma~\ref{lem: inequality}(1), and the reverse inequality by Lemma~\ref{lem: join and proximity morphism}(2).

(PM8) 
Let $0\le c$. First suppose that $0 \le a$. The inequality $\alpha(ac) \ge \alpha(a)\alpha(c)$ holds by Lemma~\ref{lem: inequality}(2), and the reverse inequality by Lemma~\ref{lem: join and proximity morphism}(3). Now apply the argument of \cite[Rem.~8.9]{BMO16} to conclude that the equality holds for all $a$. 
\end{proof}

\bibliographystyle{amsplain}
\bibliography{../Gelfand}

\end{document}